\documentclass[11pt]{article}

\usepackage[english]{babel}

\usepackage[letterpaper,top=2cm,bottom=2cm,left=3cm,right=3cm,marginparwidth=1.75cm]{geometry}

\usepackage{amsmath}
\usepackage{amsthm}
\usepackage{graphicx}
\usepackage[colorlinks=true, allcolors=blue]{hyperref}
\usepackage{amssymb}
\usepackage{comment}
\usepackage[nointlimits]{esint}
\usepackage{tikz}
\usepackage{bm}
\usepackage{authblk}

\newcommand{\RR}{\mathbb{R}}

\newcommand{\NN}{\mathbb{N}}
\newcommand{\OmegaT}{{\Omega_T}}

\newcommand{\dx}{\mathrm{dx}}

\renewcommand{\div}{\mathrm{div}}

\newcommand{\dd}{\mathrm{d}}

\newcommand{\weakstar}{\overset{\ast}{\rightharpoonup}} 
\newcommand{\weak}{\rightharpoonup} 

\newcommand{\uvec}{\mathbf{u}} 
\newcommand{\Psivec}{\mathbf{\Psi}} 

\newcommand{\coup}{\alpha} 
\newcommand{\capi}{\kappa} 
\newcommand{\paracoup}{\beta} 
\newcommand{\Peff}{P_\alpha} 
\newcommand{\system}{PNSK } 
\newcommand{\Pgrowth}{\gamma} 
\newcommand{\osc}{\mathcal{Q}} 
\newcommand{\oP}{\overline{P_\alpha}} 
\newcommand{\curlyD}{\mathcal{D}} 
\newcommand{\curlyC}{\mathcal{C}} 
\newcommand{\tgamma}{{\tilde{\gamma}}}

\newcommand{\Bog}{\mathcal{B}}

\newtheorem{theorem}{Theorem}[section]
\newtheorem*{theorem*}{Theorem}
\newtheorem{proposition}[theorem]{Proposition}
\newtheorem{corollary}[theorem]{Corollary}
\newtheorem{lemma}[theorem]{Lemma}

\theoremstyle{remark}
\newtheorem{remark}[theorem]{Remark}

\theoremstyle{definition}
\newtheorem{definition}[theorem]{Definition}

\numberwithin{equation}{section}

\title{Effective Equations for a Compressible Liquid-Vapor Flow Model with Highly Oscillating Initial Density}
\author[1]{Christian Rohde\thanks{\href{mailto:christian.rohde@mathematik.uni-stuttgart.de}{christian.rohde@mathematik.uni-stuttgart.de}}}
\author[1]{Florian Wendt\thanks{\href{mailto:florian.wendt@mathematik.uni-stuttgart.de}{florian.wendt@mathematik.uni-stuttgart.de}}}
\affil[1]{Institute of Applied Analysis and Numerical Simulation,\protect \\Pfaffenwaldring~57, D-70569 Stuttgart, Germany}
\date{}

\begin{document}
\maketitle

\begin{abstract}
We derive and justify  a new effective model for a compressible viscous liquid-vapor flow on a spray-like scale, i.e., for settings with a large  number of phase boundaries.
As a model on the detailed scale, we start from a parabolic relaxation of the Navier--Stokes--Korteweg system. We consider a sequence of initial data where the sequence of initial densities is assumed to be highly oscillating mimicking the high number of phase boundaries initially.
Then, we consider a sequence of finite energy weak solutions corresponding to the sequence of initial data.
Anticipating that the effective equations are found in the limit of infinitely many initial phase changes, we interpret the densities as Young measures and prove the convergence of the sequence of solutions to the effective model. The effective model consists of a deterministic part for the fluid's hydrodynamic quantities and a kinetic equation for the limit Young measure encoding the mixing dynamics.
By characterizing the Young measure with the corresponding cumulative distribution function, we rewrite the kinetic equation for the Young measure into a kinetic equation for the cumulative distribution function such that the resulting equations are accessible by standard approximation methods.
\end{abstract}

\vspace{2mm}

\section{Introduction}\label{Sec : Introduction}
We consider a bounded domain $\Omega \subseteq \RR^3$ with regular boundary that is occupied by a homogeneous compressible viscous fluid that can occur in a liquid and a vapor phase.
We assume that the dynamics of this  homogeneous two-phase fluid are given by the Navier--Stokes--Korteweg (NSK) equations \cite{AMW}.
These equations model the two-phase fluid with a diffuse interface, i.e., the interface between the liquid and the vapor phase is assumed to be a small region of positive Lebesgue-measure and the fluid's quantities vary possibly rapidly but smoothly over this region.
Let us assume that we are given some positive time $T \in (0,\infty)$ and let us denote $\Omega_T:=(0,T)\times\Omega$.
Then, the unknowns of the NSK equations are the fluid's density $\rho \colon [0,T) \times \Omega \to \RR_{\geq 0}$ and the fluid's velocity $\uvec \colon [0,T) \times \Omega \to \RR^3$ that obey 
\begin{equation}\label{NSK}
    \begin{aligned}
        &\partial_t \rho + \div(\rho \uvec) = 0 &&\text{in } \OmegaT,\\[0.5em]
        &\partial_t (\rho \uvec) + \div(\rho \uvec \otimes \uvec) + \nabla P(\rho) -\mu \Delta \uvec - (\lambda + \mu ) \nabla(\div \uvec) - \kappa \rho \nabla \Delta \rho = 0 &&\text{in } \OmegaT,
    \end{aligned}
\end{equation}
subject to the initial conditions
\begin{equation*}
    \rho(0,\cdot) = \rho^0,\quad \uvec(0,\cdot) = \uvec^0 \quad \text{in } \Omega,
\end{equation*}
for given functions $\rho^0 \colon \Omega \to \RR_{\geq 0}$, $\uvec^0\colon \Omega \to \RR^3$.
In the sequel we impose no-slip boundary conditions.
Note that we can deliberately allow for the density to admit vacuum states because we use the framework of finite energy weak solutions in this work (cf.\ Definition~\ref{Def:Renormalized Finite Energy Weak Solutions}).
In \eqref{NSK}, the parameters  $\mu,\lambda>0$ denote the constant viscosity coefficients, $\kappa>0$ denotes the constant capillarity coefficient and $P \colon [0,\infty) \to [0,\infty)$ denotes the pressure function.
If the pressure function $P$ satisfies $P(0)=0$ and admits two positive constants $0<r_1<r_2$ such that $P$ is monotonically decreasing on $(r_1,r_2)$ and monotonically increasing on $[0,r_1]\cup [r_2,\infty)$, then we call $P$ a pressure function of Van-der-Waals type.
For such a pressure function, we call the fluid's state vapor, spinodal and liquid, if the  density attains its value in $[0,r_1]$, $(r_1,r_2)$ and $[r_2,\infty)$, respectively.
In that sense, we consider the NSK equations to be able to describe phase transition effects if the pressure function $P$ is given of Van-der-Waals type (cf.\ Figure~\ref{P VdW Type figure}).
However, if the number of phase boundaries is very high (e.g. in a spray), the NSK model becomes computationally unfeasible due to the rapid oscillations in the fluid's quantities as for instance in the  density.
\par
In this paper, we are concerned with the rigorous justification of an effective system for a compressible viscous liquid-vapor flow described by the NSK equations with a pressure function of Van-der-Waals type in precisely that regime.
This effective system describes the two-phase fluid on a larger scale, but keeps information from the detailed scale.
As we shall see in the sequel, such an effective model can be seen as an approximation of the NSK model with a highly oscillating initial density.
One possibility to find and rigorously justify such an equation is to investigate the propagation of initial density oscillations \cite{HB_NEW,HI_P,SE,TZ}.
More specifically, we start from a sequence of initial data ${(\rho_{n}^0,u_n^0)}_{n\in\NN}$, where $n \in \NN$ should display the number of phase boundaries that we have initially.
Due to the high number of phase boundaries, we expect that the sequence of initial densities ${(\rho_n^0)}_{n\in\NN}$ is highly oscillating between the vapor's density and the liquid's density.
After constructing an appropriate corresponding sequence of solutions ${(\rho_n,\uvec_n)}_{n\in\NN}$, it is then reasonable to assume that the effective equations are found in the limit $n \to \infty$, i.e., in the limit where initially the number of phase boundaries tends to infinity.
Due to the hyperbolic character of the continuity equation $\eqref{NSK}_1$, we expect that oscillations present in the initial densities propagate in time and therefore, that the sequence of densities will not converge in any strong topology.
In order to obtain a closed limit system, the non-linear nature of the pressure function $P$ forces us to interpret the density as a Young measure. 
Thus, assuming for the moment that we are able to pass to the limit $n \to \infty$ in the NSK equations, the effective model will consist of a deterministic part for the fluid's hydrodynamic quantities and a kinetic equation for the Young measure that encodes the mixing dynamics and closes the model (cf.\ Theorem~\ref{thm: Main Result}).
\par
With this methodology, a first rigorous homogenization results was obtained in \cite{SE} for the 1D compressible Navier--Stokes equations for one-phase fluids with constant viscosity coefficient, and a result for the 3D case with constant viscosity coefficients but keeping the analysis formal.
Concerning the density dependent viscosity case, this methodology was also used in \cite{HB_NEW}.
There, the authors justified rigorously a one-velocity Baer--Nunziato model as an effective model by assuming that the Young measure is initially given by a convex combination of two Dirac-measures.
More specifically, they showed by a weak-strong uniqueness argument that the kinetic equation propagates this specific structure in time and that the kinetic equation reduces to a Baer--Nunziato model.
In \cite{LT_CONT} the authors extended the methodology to a two-component flow where each component was allowed to admit its own pressure functions and its own constant viscosity coefficient.
This problem was also tackled in \cite{LT_DISC} with a semi-discrete approach.
At this point we also refer to \cite{HSH_I,HSH_II}, where the authors justified a one-velocity Baer--Nunziato model as an effective model for a bubbly flow incorporating surface tension effects.
Concerning the 3D case with constant viscosity coefficients, similar results were obtained in \cite{HB_NOTE, BR_HU, HI_P}.
At this point we also refer to \cite{PS} for a homogenization result in terms of a cumulative distribution function and without the characterization of the Young measure as a convex combination of Dirac-measures.
\par
However, all these results do not account for phase transition effects in the detailed scale model.
In the 1D case, a homogenization result that accounts for phase transition effects in the detailed scale model was recently obtained in \cite{RW}. 
There, the detailed scale model was chosen as a non-local approximation of the NSK equations and the viscosity and capillarity coefficients were kept constant.
It is the goal of this paper to provide a similar result for the NSK equations in the 3D case generalizing in some sense the results in \cite{HI_P}.
However, the NSK equations $\eqref{NSK}$ seem hardly accessible for the investigation of the propagation of initial density oscillations due to a third-order differential operator acting on the density in the momentum equation.
To overcome this remedy, we choose a parabolic approximation of the NSK equations $\eqref{NSK}$ proposed in \cite{HKMR}, that we call the parabolic Navier--Stokes--Korteweg (PNSK) equations from now on.
This model substitutes the third-order differential operator by a first-order differential operator by introducing an additional unknown $c$ that satisfies a parabolic equation.
More precisely, in the parabolic NSK model, the fluid's dynamics are described by the density $\rho \colon [0,T) \times \Omega \to \RR_{\geq 0}$, the  velocity $\uvec \colon [0,T) \times \Omega \to \RR^3$ and the order parameter $c \colon [0,T) \times \Omega \to \RR$ that satisfy 
\begin{equation}\label{NSK_parabolic}
        \begin{aligned}
            &\partial_t\rho + \div(\rho \uvec) = 0  &&\text{in } \Omega_T,\\[0.5em]
            &\partial_t(\rho \uvec) + \div(\rho \uvec \otimes \uvec) + \nabla P(\rho)
            - \mu \Delta \uvec - (\lambda + \mu) \nabla(\div\uvec) - \coup \rho\nabla (c - \rho) = 0 
             &&\text{in } \Omega_T,\\[0.5em]
            &\paracoup \partial_t c - \capi \Delta c + \coup (c-\rho) = 0 &&\text{in } \Omega_T,
        \end{aligned}
\end{equation}
subject to initial conditions
\begin{equation}\label{NSK Parabolic Initial Conditions}
    \rho(0,\cdot) = \rho^0, \quad 
    \uvec(0,\cdot)= \uvec^0, \quad 
    c(0,\cdot)=c^0 \quad 
    \text{in } \Omega,
\end{equation}
for given functions $\rho\colon \Omega \to \RR_{\geq 0}$, $\uvec\colon \Omega \to \RR^3$, $c \colon \Omega \to \RR$ and the boundary conditions
\begin{equation}\label{NSK Parabolic Boundary Conditions}
    \uvec_{\partial \Omega} = 0 ,\quad \nabla c \cdot \mathbf{n}_{\partial\Omega } = 0
    \quad \text{on}\quad  [0,T] \times \partial\Omega.
\end{equation}
Here, the quantities $\mu,\lambda,P$ are defined as before and the coupling coefficients $\coup,\paracoup>0$ are positive constants.
Formally, we notice that in the relaxation limit $\coup \to \infty$ and $\beta \to 0$ the \system equations $\eqref{NSK_parabolic}$ approach the NSK equations $\eqref{NSK}$. 
This justifies to use $\eqref{NSK_parabolic}$ as an approximate system to $\eqref{NSK}$ if $\alpha$ and $\beta$ is sufficiently large and small, respectively.
We refer to \cite{HKMR} for a numerical investigation of the \system equations and to \cite{CRW} for a rigorous convergence result for the relaxation limit.
See also \cite{GLT} for a rigorous convergence result for $\coup\to\infty$ in the case $\paracoup=0$.
\\
Introducing the artificial pressure function 
\begin{align}\label{Definition Artificial Pressure}
    \Peff\colon [0,\infty) \to [0,\infty), \quad r \mapsto P(r) + \frac{\coup}{2} r^2,
\end{align}
we can rewrite the momentum equation $\eqref{NSK_parabolic}_2$ as
\begin{align*}
    \partial_t(\rho \uvec)
    +
    \div(\rho\uvec \otimes \uvec)
    +
    \nabla \Peff(\rho) 
    -
    \mu \Delta \uvec
    -
    (\lambda + \mu) \nabla (\div \uvec)
    -
    \coup \rho \nabla c
    =
    0.
\end{align*}
In particular, for a pressure function of Van-der-Waals type, we can choose $\coup$ large enough, so that the artificial pressure function $\Peff$ becomes monotonically increasing (cf.\ Figure~\ref{P VdW Type figure}).
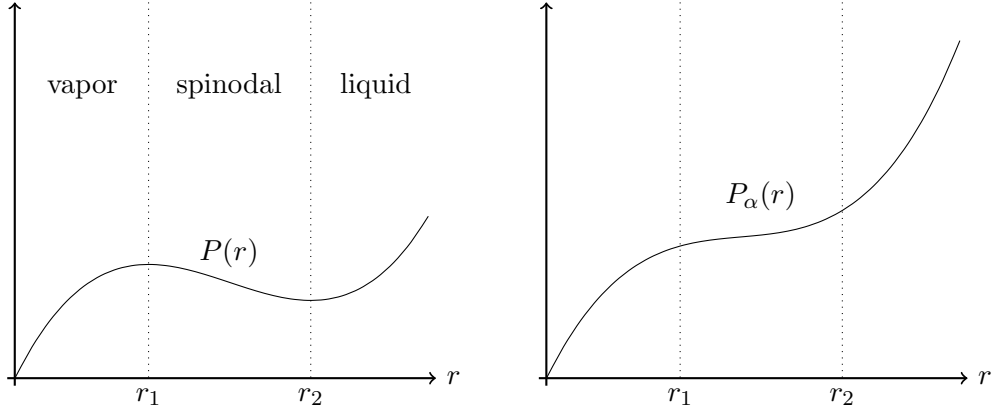
\begin{figure}
    \centering
        \begin{tikzpicture}
                \draw[->, thick](-0.1,0) -- (5.6,0) node[right] {$r$};
                \draw[->, thick](0,-0.1) -- (0,5);
                \draw[scale=3.234, domain=0:1.7, smooth, variable=\x, black] plot ({\x}, {
                (\x-0.55)^3 - (\x-0.55)^2 + 0.468
                });
                \draw[dotted] (1.7787,5.0) -- (1.7787,0) node[below] {$r_1$};
                \draw[dotted] (3.9348,5.0) -- (3.9348,0) node[below] {$r_2$} ;
                \draw (0.9,4.1) node[below] {vapor};
                \draw (2.85,4.2) node[below] {spinodal};
                \draw (4.8,4.2) node[below] {liquid};
                \draw (2.85,2) node[below] {\color{black}$P(r)$};
        \end{tikzpicture}
        \qquad 
                \begin{tikzpicture}
                \draw[->, thick](-0.1,0) -- (5.6,0) node[right] {$r$};
                \draw[->, thick](0,-0.1) -- (0,5);
                \draw[scale=3.234, domain=0:1.7, smooth, variable=\x, black] plot ({\x}, {
                (\x-0.55)^3 - (\x-0.55)^2 + 0.468 + 0.25*\x^2
                });
                \draw[dotted] (1.7787,5.0) -- (1.7787,0) node[below] {$r_1$};
                \draw[dotted] (3.9348,5.0) -- (3.9348,0) node[below] {$r_2$} ;
                \draw (2.85,2.75) node[below] {\color{black}$\Peff(r)$};
        \end{tikzpicture}
        \caption{Example of a pressure function of Van-der-Waals type. Left: the original pressure function $P$. Right: the corresponding artificial pressure function $\Peff$ for a fixed coupling coefficient $\coup>0$.}\label{P VdW Type figure}
\end{figure}
The monotonicity of $\Peff$ renders the underlying first-order system hyperbolic and thus facilitates the numerical treatment of the \system equations, see \cite{HKMR} for more details.
However, in this work, we will not require the monotonicity of the pressure function $\Peff$.
\par
Investigating the propagation of initial density oscillations under $\eqref{NSK_parabolic}$, we are able to rigorously justify a new effective model for a compressible viscous liquid-vapor flow that is modeled by the \system equations on the detailed scale. 
In this effective model, we search for a Young measure $\nu$, an order parameter $c$ and a velocity $\uvec$ that satisfy
\begin{equation}\label{Introduction: Effective Equation}
    \begin{aligned}
        &\partial_t\nu 
        +
        \div( \nu \uvec)
        +
        \partial_\xi\bigl( ( \osc - \div \uvec)\xi \nu \bigr)
        +
        \bigl( \osc - \div\uvec\bigr) \nu
        = 0
        &&\text{in } \Omega_T \times \RR,
        \\[0.5em]
        &\partial_t(\rho \uvec) 
        + 
        \div(\rho \uvec \otimes \uvec) 
        + 
        \nabla \oP 
        - 
        \mu \Delta \uvec 
        - 
        (\lambda + \mu)\nabla(\div \uvec) 
        - 
        \coup \rho \nabla c = 0
        &&\text{in } \Omega_T,
        \\[0.5em]
        &\paracoup\partial_t c 
        - 
        \capi \Delta c 
        + 
        \coup (c-\rho) 
        = 0
        \quad &&\text{in } \Omega_T,
    \end{aligned}
\end{equation}
with the closure
\begin{align*}
    \rho(t,x):= \int_{[0,\infty)} \xi \, \dd \nu_{(t,x)}(\xi) ,\quad
    \oP(t,x) := \int_{[0,\infty)} \Peff(\xi) \, \dd \nu_{(t,x)}(\xi)
\end{align*}
for almost all $(t,x) \in \Omega_T$ and
\begin{align*}
        \osc(t,x,\xi):= \frac{\oP(t,x)-\Peff(\xi)}{\lambda + 2 \mu} 
\end{align*}
for almost all $(t,x,\xi) \in \Omega_T \times \RR$ and subject to the boundary conditions $\eqref{NSK Parabolic Boundary Conditions}$ and the initial conditions
\begin{equation}\label{Introduction Effective Equations Initial Conditions}
    \uvec(0)=\uvec^0, \quad 
    c(0) = c^0, \quad 
    \nu(0) = \nu^0 \quad \text{in } \Omega,
\end{equation}
for given functions $\uvec^0\colon \Omega \to \RR^3$, $c^0 \colon \Omega \to \RR$ and a given Young measure $\nu^0 \colon \Omega \to \mathcal{P}(\RR)$ (cf.\ Definition~\ref{Def : Young Measure}).
We emphasize that the effective model $\eqref{Introduction: Effective Equation}$ is a closed system for the unknowns $(\nu,\uvec,c)$ as the density $\rho$ and the pressure function $\oP$ are completely described by the Young measure $\nu$.
\par
Our analysis justifying \eqref{Introduction: Effective Equation} as an effective model relies on an improved \emph{a priori} estimate for the artificial pressure function $\Peff$ and a weak continuity result for the effective viscous flux.
These results are known from the theory related to the compressible Navier--Stokes equations and form the heart of the proof for the global-in-time existence of finite energy weak solutions to the corresponding initial-boundary value problem (see e.g.~\cite{FNP,L}).
Since the \system equations admit the additional nonlinear contribution $\rho \nabla c$ in the momentum equation, it is \emph{a priori} unclear whether these results also hold for the \system equations.
As a mathematical novelty, we derive an improved pressure estimate (cf.\ Lemma~\ref{lem: Improved Space Time Pressure}) and the weak continuity of the effective viscous flux (cf.\ Lemma~\ref{Lemma : Preliminary for Main Proof}) for the \system equations by interpreting this additional term as a force term and exploiting parabolic regularity estimates.
\par
In order to make the kinetic equation for the Young measure accessible for standard approximation methods, we switch our point of view by considering instead of the Young measure $\nu$ its corresponding cumulative distribution function $f$ that is defined for almost all $(t,x) \in \Omega_T$ through
\begin{equation*}
    f(t,x,\xi):= \nu_{(t,x)}((-\infty,\xi]) \quad \text{for any } \xi \in \RR.
\end{equation*}
We are then able to show, that the kinetic equation on $\nu$ implies a kinetic equation for $f$ in the form
\begin{equation}\label{Intro : Kinetic Equation in Terms of f}
    \partial_t f 
    +
    \div\bigl( f \uvec\bigr) 
    -
    \partial_\xi \Bigl( \xi f \div\uvec - \xi \mathcal{M}[f] \Bigr) = 0
    \quad \text{in } \Omega_T \times \RR,
\end{equation}
with
\begin{equation}\label{Intro : Kinetic Equations in Terms of f closure M}
    \mathcal{M}[f](t,x,\xi) := \int_{[0,\xi]} \osc(t,x,\eta) \, \dd f(t,x,\eta)
\end{equation}
for almost all $(t,x,\xi) \in \Omega_T \times \RR$ and
\begin{equation}\label{Intro : Kinetic Equations in Terms of f closure rho and P}
    \rho(t,x) = \int_{[0,\infty)} \xi \, \dd f (t,x,\xi),
    \quad 
    \oP(t,x) = \int_{[0,\infty)} \Peff(\xi)  \, \dd f(t,x,\xi)
\end{equation}
for almost all $(t,x) \in \Omega_T$, where the integral with respect to $\dd f(t,x,\cdot)$ denotes the Lebesgue--Stieltjes integral with respect to $f(t,x,\cdot)$.
We emphasize that $\eqref{Introduction: Effective Equation}_2$, $\eqref{Introduction: Effective Equation}_3$ and $\eqref{Intro : Kinetic Equation in Terms of f}$ together with $\eqref{Intro : Kinetic Equations in Terms of f closure M}$ and $\eqref{Intro : Kinetic Equations in Terms of f closure rho and P}$ is a closed system for the unknowns $(f,\uvec,c)$.
\\ \par
This paper is structured as follows.
In Section~\ref{Sec : Finite Energy Weak Solutions}, we introduce and motivate the concept of finite energy weak solutions.
In particular, we introduce a class of pressure functions that allows us to describe concisely for which pressure functions our main result Theorem~\ref{thm: Main Result} applies.
In Section~\ref{Sec : Main Result}, we briefly introduce the concept of Young measures and state our main result.
In Section~\ref{Sec : Uniform Estimates} we derive appropriate \emph{a priori} estimates for the \system equations that are necessary to perform the limit procedure later on. 
Section~\ref{Sec : Proof of the Main Theorem} is devoted to the full proof of our main result Theorem~\ref{thm: Main Result}.
We close this work with a conclusion and an outlook in Section~\ref{Sec : Conclusions}.
For the sake of completeness, we provide in the Appendix~\ref{Sec : Compactness Results and Weak Continuity in Time} some known compactness results that are used in this work.

\section*{Notations}
Let $d \in \NN$ and $Q  \subseteq \RR^d$ a domain.
We denote by $C^0(Q)$ the space of continuous functions defined on $Q$.
We denote by $C^0_c(Q)$ the space of continuous functions on $Q$ with compact support and by $C^0_0(Q)$ the closure of the space $C^0_c(Q)$ under the norm 
\begin{equation*}
    \|f\|_{C^0(Q)} := \sup\limits_{x \in Q } |f(x)| \quad (f \in C^0_c(Q)).
\end{equation*}
For $k\in \NN \cup \{0\}$, we denote by $C^k(Q)$ and $C^k_c(Q)$ the space of $k$-times continuously differentiable functions and the space of $k$-times continuously differentiable functions with compact support in $Q$, respectively.
We also denote $C^\infty(Q) := \bigcap\limits_{k=0}^\infty C^k(Q)$ and $C^\infty_c(Q) := \bigcap\limits_{k=0}^\infty C^k_c(Q)$.
In the context of distributions, we sometimes write $\mathcal{D}(Q):= C^\infty_c(Q)$. 
We denote the space of distributions on $Q$ as $\mathcal{D}^\prime(Q)$.
For $p \in [1,\infty]$, we denote the Lebesgue spaces on $Q$ as $L^p(Q)$ and the $L^p(Q)$-norm as $\|\cdot\|_{L^p(Q)}$.
Thereby, we denote for some function $f \in L^1(Q)$ its mean value on $Q$ by 
\begin{equation*}
    \fint_Q f := \frac{1}{|Q|}\int_Q f \, \dd x.
\end{equation*}
We denote the $k$-th Sobolev spaces as $W^{k,p}(Q)$ and the $W^{k,p}(Q)$-norm as $\|\cdot\|_{W^{k,p}(Q)}$. 
For $q \in [1,\infty)$, we denote by $W^{k,q}_0(Q)$ the closure of the space $C^\infty_c(Q)$ under the norm $\|\cdot\|_{W^{k,p}(Q)}$.
If $q = 2$, we sometimes write $H^k(Q)$ and $H^k_0(Q)$ instead of $W^{k,2}(Q)$ and $W^{k,2}_0(Q)$, respectively.
For $s \in (1,\infty)$, we denote by $W^{-1,s}(Q)$ the dual space of $W^{1,s^\prime}_0(Q)$, where $s^\prime$ denotes the conjugate Hölder exponent of $s$.
To ease the notation, we use sometimes for vector-valued or matrix-valued function spaces the same notation as in the scalar case. 
For instance, we write $L^p(Q)$ instead of $L^p(Q;\RR^3)$ if no ambiguities appear.
For some Banach space $X$, we denote the Bochner spaces on $Q$ ranging into $X$ as $L^p(Q;X)$ and its Bochner norms as $\|\cdot\|_{L^p(Q;X)}$.
Let $X^*$ be the dual space to $X$. Then we denote the space of all weakly-$*$ measurable $L^\infty$-functions on $Q$ mapping into $X^*$ as
\begin{align*}
    L^\infty_{\mathrm{w}^\ast}(Q;X^*):= 
    \biggl\{ f\colon Q \to X^*\Big|
    \,\,&Q\ni x \mapsto \langle f(x),v\rangle_{{X^\ast},\,X}\,\,\, \text{is Lebesgue-measurable for any } v \in X,\\
    &\Bigl[x \mapsto \|f(x)\|_{X^*}\Bigr] \in L^\infty(Q)
    \biggr\}. 
\end{align*}
The norm on $L^\infty_{\mathrm{w}^*}(Q;X^*)$ is denoted by
\begin{equation*}
    \|f\|_{L^\infty_{\mathrm{w}^*}(Q;X^*)} := \mathrm{ess}\;\mathrm{sup}_{x \in Q} \|f(x)\|_{X^*} 
\end{equation*}
for $f \in L^\infty_{\mathrm{w}^*}(Q;X^*)$, and for a sequence ${(f_n)}_{n\in\NN} \subseteq L^\infty_{\mathrm{w}^*}(Q;X^*)$, we write
\begin{equation*}
    f_n \weakstar f \quad \text{in }  L^\infty_{\mathrm{w}^*}(Q;X^*) \quad \text{as } n\to\infty,
\end{equation*}
if 
\begin{equation*}
    \int_Q\langle f_n(x), \varphi(x)\rangle_{X^*,\,X}\,\, \dd x \xrightarrow{n\to\infty} \int_Q\langle f(x),\varphi(x)\rangle_{X^*,\,X}\, \,\dd x
\end{equation*}
for any test function $\varphi \in L^1(Q;X)$.
For some positive time $T>0$, we define the space of weakly continuous functions on $[0,T]$ ranging into $L^q(Q)$ as
\begin{align*}
    C_{\mathrm{w}}([0,T];L^q(Q)) 
    :=
    \biggl\{
        f \colon [0,T] \to L^q(Q) \Big| &\Bigl[t\mapsto \int_Q f(t)\phi\, \dd x\Bigr]\in C([0,T]) 
        \,\,\,\text{for any} \,\,\,
        \phi \in L^{q^\prime}(Q), \\
        &\sup\limits_{t\in[0,T]}\|f(t)\|_{L^q(Q)} <\infty \biggr\},
\end{align*}
where $q^\prime \in (1,\infty]$ denotes the conjugate Hölder exponent to $q$.
For a sequence $(f_n) \subseteq  C_{\mathrm{w}}([0,T];L^q(Q))$ and $f \in C_{\mathrm{w}}([0,T];L^q(Q))$, we write
\begin{equation*}
    f_n \xrightarrow{n\to\infty} f \quad \text{in } C_{\mathrm{w}}([0,T];L^q(Q))
\end{equation*}
if
\begin{equation*}
    \sup\limits_{t\in[0,T]} \biggl| \int_Q f_n(t)\phi\, \dd x - \int_Q f(t)\phi \, \dd x \biggr| \xrightarrow{n\to\infty} 0
    \quad \forall \, \phi \in L^{q^\prime}(Q).
\end{equation*}
Finally, we denote by $\mathcal{M}(\RR)$ the space of all signed finite Radon measures on $\RR$ that can be identified with the dual space $\bigl( C^0_0(\RR) \bigr)^*$ and by $\mathcal{P}(\RR)$ the set of all probability measures on $\RR$.

\section{Admissible Pressure Functions and Weak Solutions}\label{Sec : Finite Energy Weak Solutions}
We shall work with the concept of finite energy weak solutions.
The global-in-time existence of finite energy weak solutions was shown for the isentropic compressible Navier--Stokes equations with constant viscosity coefficients in \cite{L} for an adiabatic pressure exponent $\gamma > \frac{9}{5}$ in spatial dimension three.
These results were then extended in \cite{FNP} to pressure exponents $\gamma > \frac{3}{2}$ and in \cite{F} for a  class of non-monotone pressure functions.
In all these results, the growth rate of the pressure function towards infinity plays a crucial role, as it provides an \emph{a priori} estimate for the density in an appropriate Lebesgue space.
In particular, this growth rate is an integral part of the definition of finite energy weak solutions.
Thus, before introducing the concept of finite energy weak solutions for the \system model, let us introduce the following definition with that we describe in the sequel concisely for which class of pressure functions our homogenization result applies.

\begin{definition}[Admissible pressure function]
    We call a function $P \in C^1([0,\infty))$ admissible with growth rate $\Pgrowth\in(1,\infty)$ if
    \begin{equation}\label{Definition Admissible Pressure Function Growth Rate}
        P\geq 0,
        \quad 
        P(0) = 0,
        \quad 
        \lim\limits_{r \to \infty}\frac{P^\prime(r)}{r^{\Pgrowth-1}} = P_\infty,
    \end{equation}
    for some positive constant $P_\infty >0$.
\end{definition}
\begin{remark}
    The class of admissible pressure functions accounts for a broad variety of monotone and non-monotone pressure functions. 
    In particular, any pressure function of Van-der-Waals type satisfying the third condition in $\eqref{Definition Admissible Pressure Function Growth Rate}$ for some $\gamma \in (1,\infty)$ is admissible with growth rate $\gamma$.
\end{remark}
For some admissible pressure function $P\colon[0,\infty) \to [0,\infty)$ with growth rate $\gamma \in (1,\infty)$, we assign a pressure potential $W\colon [0,\infty)\to[0,\infty)$ via
\begin{align}\label{Definition Pressure Potential}
    W(r):= r\int_1^r \frac{P(z)}{z^2}\, \dd z
    \qquad \text{for } z\in[0,\infty).
\end{align}
We readily verify that $W\in C^0([0,\infty))\cap C^2((0,\infty))$ and that
\begin{align}\label{Introdution Relation Presure Potential}
    P(r)=W^\prime(r)r-W(r) \qquad \text{for } r \in (0,\infty).
\end{align}
Furthermore, due to the third relation in $\eqref{Definition Admissible Pressure Function Growth Rate}$, we have that $W$ satisfies the bounds
\begin{align}\label{bound pressure potential}
    r^\gamma \leq c_1 + c_2 W(r),
    \quad 
    |W(r)|\leq c_3 + c_4 r^\gamma
    \quad \text{for } r \in [0,\infty),
\end{align}
and some positive constants $c_1,c_2,c_3,c_4\in (0,\infty)$ that do not depend on $r$.
By the same token, we have that $\Peff$ defined as in $\eqref{Definition Artificial Pressure}$ satisfies for $\tgamma:=\max\{2,\gamma\}$ the bounds
\begin{align}\label{bound for Peff}
    \Peff(r) \leq c_5 + c_6r^\tgamma, 
    \quad 
    r^\tgamma \leq c_7 + c_8 \Peff(r)
    \qquad \text{for } r \in [0,\infty),
\end{align}
and some positive constants $c_5,c_6,c_7,c_8>0$ that do not depend on $r$.\par
Let us now introduce the concept of finite energy weak solutions for the initial-boundary value problem to the \system equations $\eqref{NSK_parabolic}$--$\eqref{NSK Parabolic Boundary Conditions}$.
To motivate the definition of finite energy weak solutions, we identify the energy dissipation and derive certain \emph{a priori} bounds in terms of the initial data in a smooth framework.
Let us assume that we are given some classical solution $(\rho,\uvec,c)$ to $\eqref{NSK_parabolic}$--$\eqref{NSK Parabolic Boundary Conditions}$ emanating from the initial data $(\rho^0,\uvec^0,c^0)$ and existing on $[0,T]$.
First, we obtain for any $b \in C^1(\RR)$ from the continuity equation $\eqref{NSK_parabolic}_1$ that 
\begin{align}\label{Introduction Renormalized Continuity Equation}
    \partial_t b(\rho) + \div(b(\rho) \uvec) + \bigl( b^\prime(\rho) \rho - b(\rho) \bigr) \div \uvec = 0
\end{align}
holds in $(0,T)\times\Omega$.
Moreover, multiplying the momentum equation $\eqref{NSK_parabolic}_2$ by the velocity $\uvec$, integrating in space over $\Omega$ using the continuity equation $\eqref{NSK_parabolic}_1$ and integration by parts yields for any $t \in (0,T)$ the relation
\begin{align}\label{Formal Derivation Energy Dissipation Eq1}
    &\frac{\dd}{\dd t}\int_\Omega \frac{1}{2}\rho |\uvec|^2 \, \dd x
    -
    \int_\Omega P(\rho) \div \uvec \, \dd x
    -
    \int_\Omega \coup \partial_t \rho (c - \rho) \, \dd x\nonumber\\
    &\qquad+
    \int_\Omega \bigl(\mu \nabla \uvec : \nabla \uvec + (\lambda + \mu) |\div\uvec|^2 \bigr) \, \dd x 
    =
    0.
\end{align}
Relation $\eqref{Introdution Relation Presure Potential}$, $\eqref{Introduction Renormalized Continuity Equation}$,  and the boundary conditions $\eqref{NSK Parabolic Boundary Conditions}$ imply
\begin{align}\label{Formal Derivation Energy Dissipation Eq2}
    -\int_\Omega P(\rho) \div \uvec \, \dd x 
    = \frac{\dd}{\dd t} \int_\Omega W(\rho) \, \dd x.
\end{align}
Using the parabolic equation $\eqref{NSK_parabolic}_3$, the boundary conditions $\eqref{NSK Parabolic Boundary Conditions}$ and integration by parts leads to
\begin{align}\label{Formal Derivation Energy Dissipation Eq3}
    -\int_\Omega \coup \partial_t \rho (c-\rho)\, \dd x
    &= \int_\Omega \coup \partial_t(\rho- c) (\rho - c) \, \dd x
    - \int_\Omega  \partial_t c  \coup (c - \rho) \, \dd x\nonumber\\
    &=\frac{\dd}{\dd t} \int_\Omega \frac{\coup}{2} |\rho-c|^2 \, \dd x
    +\int_\Omega \paracoup |\partial_t c|^2 \, \dd x
    - \int_\Omega \capi \partial_t c \Delta c\, \dd x\nonumber\\
    &=
    \frac{\dd}{\dd t} \int_\Omega\Bigl( \frac{\coup}{2}|\rho-c|^2 + \frac{\capi}{2}|\nabla c|^2\Bigr) \, \dd x 
    + \int_\Omega \paracoup |\partial_t c|^2 \, \dd x.
\end{align}
Combining $\eqref{Formal Derivation Energy Dissipation Eq1}, \eqref{Formal Derivation Energy Dissipation Eq2}$ and $\eqref{Formal Derivation Energy Dissipation Eq3}$ yields the energy dissipation
\begin{align}\label{Introduction Energy Dissipation}
    \frac{\dd}{\dd t} E(t)
    + \int_0^t \int_\Omega \bigl( \mu \nabla \uvec : \nabla \uvec + (\lambda + \mu) |\div\uvec|^2 \bigr) \, \dd x \, \dd \tau
    + \int_0^t \int_\Omega \paracoup |\partial_t c|^2 \, \dd x \, \dd \tau = 0
\end{align}
for any $t \in (0,T)$ with the energy
\begin{align}\label{Introduction Energy}
    E(t) := \int_\Omega \Bigl( \frac{1}{2} \rho(t) |\uvec(t)|^2 + W(\rho(t)) + \frac{\coup}{2}|\rho(t)-c(t)|^2 + \frac{\capi}{2} |\nabla c(t)|^2 \Bigr) \, \dd x.
\end{align}
In $\eqref{Introduction Energy}$ we have used the notation for Bochner spaces on $(0,T)$ for the functions $\rho,\uvec$ and $c$ which are defined on the space-time cylinder $\OmegaT$.
We emphasize at this point that $\eqref{Introduction Energy Dissipation}$ implies that the \system system is thermodynamically admissible.
Integrating further relation $\eqref{Introduction Energy Dissipation}$ for fixed $t \in (0,T)$ over $[0,t]$ yields
\begin{align}\label{Introduction Int Energy Dissipation}
    &E(t) 
    + \int_0^t \int_\Omega \bigl( \mu \nabla \uvec : \nabla \uvec + (\lambda + \mu) |\div \uvec|^2 \bigr) \, \dd x \, \dd \tau 
    + \int_0^t \int_\Omega \paracoup |\partial_t c|^2 \, \dd x \, \dd \tau 
    = E_0
\end{align}
with 
\begin{align}\label{Introduction Initial Energy}
    E_0 := \int_\Omega \Bigl( \frac{1}{2} \rho^0 |\uvec^0|^2 + W(\rho^0) + \frac{\coup}{2} |\rho^0-c^0|^2 + \frac{\capi}{2} |\nabla c^0|^2 \Bigr)\, \dd x .
\end{align}
In the following estimates we denote by $\mathcal{I}_0>0$ a generic positive constant that may vary from line to line but only depends on the quantities
\begin{align*}
    \|\rho^0\|_{L^\tgamma(\Omega)},
    \,
    \|\uvec^0\|_{L^{\frac{2\tgamma}{\tgamma-1}}(\Omega)},
    \,
    \|c^0\|_{H^1(\Omega)},
    \,
    \mu,
    \,
    \lambda,
    \,
    \capi,
    \,
    \coup,
    \,
    \paracoup,
    \,
    T,
    \,
    |\Omega|,
\end{align*}
where $\tgamma:=\max\{2,\gamma\}$.
By $\eqref{bound pressure potential}$ and Hölder's inequality, we readily verify that
\begin{align*}
    E_0\leq \mathcal{I}_0.
\end{align*}
From $\eqref{bound pressure potential}$ and $\eqref{Introduction Int Energy Dissipation}$ we deduce 
\begin{align}
    &\|\sqrt{\rho}\uvec\|_{L^\infty(0,T;L^2(\Omega))}
    +
    \|\rho\|_{L^\infty(0,T;L^\gamma(\Omega))}
    +
    \|\rho-c\|_{L^\infty(0,T;L^2(\Omega))}
    +
    \|\nabla c\|_{L^\infty(0,T;L^2(\Omega))}\nonumber
    \\
    &\qquad +
    \|\uvec\|_{L^2(0,T;H^1(\Omega))}
    +
    \|\partial_tc\|_{L^2(\OmegaT)}
    \leq 
    \mathcal{I}_0.\label{Introduction A Priori Bds I}
\end{align}
Integrating the parabolic equation $\eqref{NSK_parabolic}_3$ in space over $\Omega$ and using the boundary conditions $\eqref{NSK Parabolic Boundary Conditions}$ yields for any $t \in (0,T)$ that
\begin{align*}
    \frac{\dd}{\dd t} \int_\Omega c(t)\, \dd x
    =
    -
    \frac{\coup}{\paracoup}
    \int_\Omega c(t) \, \dd x 
    +
    \frac{\coup}{\paracoup} \int_\Omega \rho(t) \, \dd x
    =
    -\frac{\coup}{\paracoup} \int_\Omega c(t) \, \dd x 
    +
    \frac{\coup}{\paracoup} \int_\Omega \rho^0 \, \dd x,
\end{align*}
where we have used the continuity equation to obtain the last identity.
From this relation, we deduce
\begin{align*}
    \biggl|\int_\Omega c(t) \, \dd x \biggr| \leq \mathcal{I}_0,
\end{align*}
which yields in combination with $\eqref{Introduction A Priori Bds I}$, Poincaré's inequality and Hölder's inequality
\begin{align}\label{Introduction A Priori Bds II}
    \|c\|_{L^\infty(0,T;H^1(\Omega))}
    +
    \|\rho\|_{L^\infty(0,T;L^\tgamma(\Omega))}
    +
    \|\rho\uvec\|_{L^\infty(0,T;L^{\frac{2\tgamma}{\tgamma+1}}(\Omega))}
    \leq \mathcal{I}_0.
\end{align}
Combining this estimate with standard parabolic regularity results (see e.g.~\cite[Chapter 7]{EPDE}), we obtain
\begin{align}\label{Introduction A Priori Bds III}
    \|c\|_{L^2(0,T;H^2(\Omega))} \leq \mathcal{I}_0.
\end{align}
The renormalized continuity equation $\eqref{Introduction Renormalized Continuity Equation}$, the energy dissipation $\eqref{Introduction Energy Dissipation}$ and the \emph{a priori} estimates $\eqref{Introduction A Priori Bds I}$--$\eqref{Introduction A Priori Bds III}$ motivate the following definition of finite energy weak solutions for the initial-boundary value problem to the \system equations $\eqref{NSK_parabolic}$--$\eqref{NSK Parabolic Boundary Conditions}$.
This definition corresponds to the definition of finite energy weak solutions for the initial-boundary value problem to the compressible Navier--Stokes equations, see e.g.~\cite{FNP,L}.
\begin{definition}[Finite energy weak solutions]\label{Def:Renormalized Finite Energy Weak Solutions}
    Let $\Omega\subseteq \RR^3$ denote a bounded domain with regular boundary and let $T,\mu,\lambda,\capi,\coup,\paracoup>0$. Suppose that $P$ is admissible with growth rate $\Pgrowth \in (1,\infty)$ and let $W \colon[0,\infty) \to [0,\infty)$ be defined as in $\eqref{Definition Pressure Potential}$.
    Assume that we are given some initial data
    \begin{align}\label{defi:renormalized:IC}
        \rho^0 \in L^{\tilde{\gamma}}(\Omega), \quad \rho^0\geq0 \quad \text{a.e.},\quad \uvec^0 \in L^{\frac{2\tgamma}{\tgamma-1}}(\Omega;\RR^3), \quad c^0 \in H^1(\Omega),
    \end{align}
    where $\tilde{\gamma}:=\max\{2,\gamma\}$.
    Then we call $(\rho,\uvec,c)$ a finite energy weak solution to $\eqref{NSK_parabolic}$--$\eqref{NSK Parabolic Boundary Conditions}$ existing on $[0,T]$ and emanating from the initial data $(\rho^0,\uvec^0,c^0)$, if the following holds:
    \begin{enumerate}
        \item We have the regularity
        \begin{alignat*}{2}
            &\rho \in C_{\mathrm{w}}([0,T];L^\tgamma(\Omega)), \quad &&\rho\geq 0 \quad \text{a.e.,}\\
            &\uvec \in L^2(0,T;H^1_0(\Omega;\RR^3)), \quad &&\rho\uvec \in C_{\mathrm{w}}([0,T];L^{\frac{2\tgamma}{\tgamma + 1}}(\Omega;\RR^3)),\\
            &c \in L^2(0,T;H^2(\Omega))\cap C([0,T];H^{1}(\Omega)), \quad &&\partial_t c \in L^2(\OmegaT),
        \end{alignat*}
        \item
        the continuity equation $\eqref{NSK_parabolic}_1$ is satisfied in the sense of integral identities
        \begin{equation}
        \begin{aligned}
            \int_0^\tau \int_\Omega \bigl( \rho \partial_t\varphi + \rho \uvec \cdot \nabla \varphi \bigr) \, \dd x \, \dd t 
            = 
            \int_\Omega \rho(\tau,\cdot)\varphi(\tau,\cdot)\, \dd x
            -\int_\Omega \rho^0 \varphi(0,\cdot) \, \dd x 
        \end{aligned}\label{Def: Renormalized Finite Energy Weak Solutions Continuity Eq}
        \end{equation}
        for any $\tau \in [0,T]$ and any test function $\varphi \in C^1_c([0,T]\times \RR^3)$,
        \item the momentum equation $\eqref{NSK_parabolic}_2$ is satisfied in the sense of integral identities
        \begin{equation}
        \begin{aligned}
            &\int_0^\tau \int_\Omega \bigl( \rho \uvec \cdot \partial_t \Psivec 
            + \rho \uvec \otimes \uvec : \nabla \Psivec 
            + \Peff(\rho) \div \Psivec \bigr)\, \dd x \, \dd t
            \\
            &\qquad- 
            \int_0^\tau \int_\Omega
            \bigl(
            \mu \nabla \uvec : \nabla \Psivec  
            +
            (\mu + \lambda ) \div \uvec \: \div \Psivec
            - 
            \coup \rho \nabla c \cdot \Psivec 
            \bigr)
            \, \dd x  \, \dd t 
            \\
            &=
            \int_\Omega (\rho\uvec)(\tau,\cdot) \cdot\Psivec(\tau,\cdot) \, \dd x
            -\int_\Omega \rho^0 \uvec^0 \cdot \Psivec(0,\cdot) \,\dx
        \end{aligned}\label{Def: Renormalized Finite Energy Weak Solutions Momentum Eq}
        \end{equation}
        for any $\tau \in [0,T]$ and any test function $\Psivec \in C^1_c([0,T]\times\Omega;\RR^3)$,
        \item the parabolic equation $\eqref{NSK_parabolic}_3$ is satisfied in the sense of integral identities
        \begin{equation}
        \begin{aligned}
            &\int_0^\tau \int_\Omega 
            \bigl(
            \paracoup c \partial_t \varphi 
            - \capi \nabla c \cdot \nabla \varphi 
            - \coup (c-\rho) \varphi 
            \bigr)\, \dd x \, \dd  t
            \\
            &=
            \int_\Omega \paracoup c(\tau,\cdot) \varphi(\tau,\cdot) \, \dd x
            -\int_\Omega \paracoup c^0 \varphi(0,\cdot)\,\dd x
        \end{aligned}\label{Def: Renormalized Finite Energy Weak Solutions Parabolic Eq}
        \end{equation}
        for any $\tau \in [0,T]$ and any test function $\varphi \in C^1_c([0,T]\times \RR^3)$,
        \item the renormalized continuity equation $\eqref{Introduction Renormalized Continuity Equation}$ is satisfied in the sense of integral identities 
        \begin{equation}
        \begin{aligned}
            &\int_0^T \int_\Omega
            \Bigl(
            b(\rho) \partial_t\varphi 
            + b(\rho)\uvec \cdot \nabla\varphi 
            - \bigl( b^\prime(\rho)\rho - b(\rho) \bigr)  \div\uvec \varphi 
            \Bigr)\, \dd x \, \dd t\\
            &\qquad =
            -\int_\Omega b(\rho^0) \varphi(0,\cdot) \, \dd x
        \end{aligned}\label{Def: Renormalized Finite Energy Weak Solutions Renormalized Continuity Eq}
        \end{equation}
        for any test function $\varphi \in C^1_c([0,T)\times\RR^3)$ and any function
        \begin{align*}
            b \in \mathcal{F}:=\Bigl\{ b \in C^1(\RR) \mid \exists M_b>0 \, : \, b^\prime(z) = 0 \quad \forall \, z \in [M_b,\infty) \Bigr \},
        \end{align*}
        \item the energy inequality
        \begin{equation}
        \begin{aligned}
            E(t) 
            + \int_0^t \int_\Omega \bigl(\mu \nabla \uvec : \nabla \uvec  + (\lambda + \mu) |\div\uvec|^2 \bigr) \, \dd x \, \dd \tau
            + \int_0^t \int_\Omega \paracoup |\partial_t c|^2 \, \dd x \, \dd \tau
            \leq E_0
        \end{aligned}\label{Def: Renormalized Finite Energy Weak Solutions Energy Eq}
        \end{equation}
        holds for almost all $t \in (0,T)$, where $E(t)$ and $E_0$ are defined as in $\eqref{Introduction Energy}$ and $\eqref{Introduction Initial Energy}$, respectively.
    \end{enumerate}
\end{definition}

The global-in-time existence of finite energy weak solutions to the initial-boundary value problem of the compressible Navier--Stokes equations with constant viscosity coefficients depends crucially on the asymptotic growth rate $\gamma$ of the pressure function which is intimately related to the available \emph{a priori} bounds for the density.
For an isentropic pressure law $r \mapsto r^\gamma$, the global-in-time existence was proved in \cite{FNP} for the regime $\gamma \in (\frac{3}{2},\infty)$, while the regime $\gamma \in (1,\frac{3}{2})$ is an open problem.
For the \system equations, the two additional quadratic contributions in the energy functional (cf.\ $\eqref{Introduction Energy}$) provide us for the whole regime $\gamma \in (1,\infty)$ an appropriate \emph{a priori} estimate for the $L^\infty(0,T;L^2(\Omega))$-norm of the density (cf.\ $\eqref{Introduction A Priori Bds II}$).
Therefore, we expect that the global-in-time existence of finite energy weak solutions to the initial-boundary value problem of the \system equations $\eqref{NSK_parabolic}$--$\eqref{NSK Parabolic Boundary Conditions}$ can be obtained for the whole regime $\gamma \in (1,\infty)$ by a non-trivial modification of the proof in \cite{FNP}.
However, providing such a global-in-time existence result would be out of scope of this paper.
Thus, in the homogenization result, we will postulate the global-in-time existence of finite energy weak solutions to the initial-boundary value problem of the \system equations $\eqref{NSK_parabolic}$--$\eqref{NSK Parabolic Boundary Conditions}$ subject to appropriate initial conditions (cf.\ Theorem~\ref{thm: Main Result}).

\section{Main Result}\label{Sec : Main Result}
In this section, we state our main result, which consists of a homogenization result for a sequence of finite energy weak solutions to the initial-boundary value problem $\eqref{NSK_parabolic}$--$\eqref{NSK Parabolic Boundary Conditions}$ with highly oscillating initial densities.
In order to state it, we need the concept of Young measures to deal with the strong oscillations that we expect in the sequence of densities.
For a detailed exposition on the topic of Young measures or parametrized measures we refer to \cite{PRG}.
For the sake of clarity, we describe in the following definition precisely what we understand by a Young measure in this work.

\begin{definition}[Young measure]\label{Def : Young Measure}
    Let $d \in \NN$ and let $Q \subseteq \RR^d$ be a domain.
    Then we call $\nu$ a Young measure on $\RR$ parametrized over $Q$ if $\nu$ lies in the set
    \begin{equation*}
        \mathcal{Y}(Q;\RR) 
        := 
        \Bigl\{ \nu \in L^\infty_{\mathrm{w}^*}(Q;\mathcal{M}(\RR)) \mid \nu_x \in \mathcal{P}(\RR) \text{ for almost all } x \in Q\Bigr\}. 
    \end{equation*}
\end{definition}

Some integrable function $f \in L^1(Q)$, where $Q \subseteq \RR^d$ denotes a domain, induces a Young measure $\nu_f \in \mathcal{Y}(Q;\RR)$ via
\begin{equation*}
    (\nu_f)_x := \delta_{f(x)} \quad \text{for a.e.} \quad x \in \Omega,
\end{equation*}
where $\delta_{f(x)}$ denotes the Dirac-measure concentrated in $f(x)$, see e.g.~\cite{PRG}.
For the rest of this work, we will keep the notation $\nu_f$ for the Young measure induced by some integrable function $f\in L^1(Q)$ as above.
In order to ease the notation in the sequel, we denote for some Young measure $\nu \in \mathcal{Y}(Q;\RR)$ and some test function $b \in C^0(\RR)$ by $\langle \nu, b\rangle$ the function
\begin{equation*}
    Q\ni x \mapsto \langle \nu_x,b\rangle = \int_\RR b(\xi) \, \dd \nu _x (\xi).
\end{equation*}
Now we are ready to state the main result of this paper.

\begin{theorem}\label{thm: Main Result}
    Let $T,\mu,\lambda,\capi,\coup,\paracoup>0$ and let $\Omega \subseteq \RR^3$ be a bounded domain with regular boundary.
    Assume that $P$ is admissible with growth rate $\Pgrowth \in (1,\infty)$ and let $\tgamma:=\max\{2,\gamma\}$.
    For $n \in \NN$, let initial data 
    \begin{align*}
        \rho_n^0 \in L^\tgamma(\Omega), \quad 
        \rho_n^0 \geq 0 \quad \text{a.e.}, \quad
        \uvec_n^0 \in L^{\frac{2\tgamma}{\tgamma-1}}(\Omega;\RR^3), \quad
        c^0_n \in H^1(\Omega)
    \end{align*}
    be given such that
    \begin{align}\label{thm: Main Result Result 1}
        \sup\limits_{n \in \NN} \|\rho^0_n\|_{L^\tgamma(\Omega)}
        +
        \sup\limits_{n \in \NN}\|\uvec^0_n\|_{L^{\frac{2\tgamma}{\tgamma-1}}(\Omega)}
        +
        \sup\limits_{n \in \NN}\|c^0_n\|_{H^1(\Omega)} < \infty,
    \end{align}
    and suppose that there exists a Young measure $\nu^0\in \mathcal{Y}(\Omega;\mathcal{M}(\RR))$ and functions $\uvec^0 \in L^{\frac{2\tgamma}{\tgamma-1}}(\Omega;\RR^3)$, $c^0 \in H^1(\Omega)$ such that
    \begin{align*}
        \nu_{\rho_n^0} \weakstar \nu^0 \quad \text{in } L^\infty_{\mathrm{w}^\ast}(\Omega;\mathcal{M}(\RR)), \quad
        \uvec_n^0 \to \uvec^0 \quad \text{in } L^{\frac{2\tgamma}{\tgamma-1}}(\Omega;\RR^3),\quad
        c^0_n \weak c^0 \quad \text{in } H^1(\Omega).
    \end{align*}
    For $n \in \NN$, let $(\rho_n,\uvec_n,c_n)$ be a finite energy weak solution to the initial-boundary value problem $\eqref{NSK_parabolic}$--$\eqref{NSK Parabolic Boundary Conditions}$ existing on $[0,T]$ and emanating from the initial data $(\rho^0_n,\uvec^0_n,c^0_n)$.\\
    Then there exist a Young measure $\nu \in \mathcal{Y}(\Omega_T;\mathcal{M}(\RR))$ and functions $\uvec \in L^2(0,T;H^{1}_0(\Omega;\RR^3))$, $c \in L^2(0,T;H^{2}(\Omega))\cap C([0,T];H^{1}(\Omega))$, $\rho \in C_{\mathrm{w}}([0,T];L^\tgamma(\Omega))$, $\oP\in L^\delta(\OmegaT)$ for some $\delta \in (1,\infty)$ with
    $
        \partial_t c \in L^2(\OmegaT),\rho\uvec \in C_{\mathrm{w}}([0,T];L^{\frac{2\tgamma}{\tgamma+1}}(\Omega))
    $
    and
    \begin{equation*}
        \rho(t,x)\geq 0, \quad \mathrm{spt}\, \nu_{(t,x)} \subseteq [0,\infty)
    \end{equation*}
    for almost all $(t,x) \in \Omega_T$, such that, after passing to a non-relabeled subsequence,
    \begin{align*}
        &\nu_{\rho_n} \weakstar \nu \quad \text{in } L^\infty_{\mathrm{w}^*}(\Omega_T; \mathcal{M}(\RR)), 
        \quad 
        \rho_n \to \rho \quad \text{in } C_{\mathrm{w}}([0,T];L^\tgamma(\Omega)), \\
        &\uvec_n \weak \uvec \quad \text{in } L^2(0,T;H^{1}_0(\Omega;\RR^3)),
        \quad 
        c_n \to c \quad \text{in } L^2(0,T;H^{1}(\Omega)),\\ 
        &\rho_n\uvec_n \to \rho \uvec \quad \text{in } C_{\mathrm{w}}([0,T];L^{\frac{2\tgamma}{\tgamma+1}}(\Omega;\RR^3)),
        \quad 
        \Peff(\rho_n) \weak \oP \quad \text{in } L^{\delta}(\OmegaT),\\
        &\rho_n\nabla c_n \weak \rho \nabla c \quad \text{in } L^2(0,T;L^\frac{6\tgamma}{\tgamma+6}(\Omega;\RR^3)), 
        \quad 
        \rho^0_n \weak \rho^0 \quad \text{in } L^\tgamma(\Omega).
    \end{align*}
    Moreover, we have for almost all $(t,x) \in \OmegaT$ that the functions $\rho$ and $\oP$ satisfy
    \begin{align*}
        \rho(t,x) = \int_{[0,\infty)} \xi \, \dd \nu_{(t,x)}(\xi), \quad \oP(t,x) = \int_{[0,\infty)} \Peff(\xi) \, \dd \nu_{(t,x)}(\xi)
    \end{align*}
    and for almost all $x \in \Omega$ we have
    \begin{align*}
        \mathrm{spt}\, \nu^0_x \subseteq [0,\infty),\quad \rho^0(x) = \int_{[0,\infty)} \xi\, \dd \nu^0(x).
    \end{align*}
    The triple $(\rho,\uvec,c)$ satisfies the integral identities
        \begin{equation}
        \begin{aligned}
            \int_0^\tau \int_\Omega 
            \bigl(
            \rho \partial_t\varphi + \rho \uvec \cdot \nabla \varphi 
            \bigr)
            \, \dd x \, \dd t 
            = 
            \int_\Omega \rho(\tau,\cdot)\varphi(\tau,\cdot)\, \dd x
            -\int_\Omega \rho^0 \varphi(0,\cdot) \, \dd x 
        \end{aligned}\label{thm: Main Result Continuity Eq}
        \end{equation}
        for any $\tau \in [0,T]$ and any test function $\varphi \in C^1_c([0,T]\times \RR^3)$,
    \begin{equation}
        \begin{aligned}
            &\int_0^\tau \int_\Omega
            \bigl(
            \rho \uvec \cdot \partial_t \Psivec 
            + \rho \uvec \otimes \uvec : \nabla \Psivec 
            + \oP \div \Psivec
            - \mu \nabla \uvec : \nabla \Psivec 
            - (\mu + \lambda ) \div \uvec \: \div \Psivec
            \bigr)
            \, \dd x \, \dd t\\
            &\qquad+ \int_0^\tau \int_\Omega  \coup \rho \nabla c \cdot \Psivec \, \dd x  \, \dd t 
            =
            \int_\Omega (\rho\uvec)(\tau,\cdot) \cdot \Psivec(\tau,\cdot) \, \dd x
            -\int_\Omega \rho^0 \uvec^0 \cdot \Psivec(0,\cdot) \,\dd x
        \end{aligned}\label{thm: Main Result Momentum Eq}
    \end{equation}
    for any $\tau \in [0,T]$ and any test function $\Psivec \in C^1_c([0,T]\times \Omega)$,
    \begin{equation}
        \begin{aligned}
            \int_0^\tau \int_\Omega 
            \bigl(
            \paracoup c \partial_t \varphi 
            - \capi \nabla c \cdot \nabla \varphi 
            - \coup (c-\rho) \varphi
            \bigr)
            \, \dd x \, \dd  t
            =
            \int_\Omega \paracoup c(\tau,\cdot) \varphi(\tau,\cdot) \, \dd x
            -\int_\Omega \paracoup c^0 \varphi(0,\cdot)\, \dd x
        \end{aligned}\label{thm: Main Result Parabolic Eq}
        \end{equation}
        for any $\tau \in [0,T]$ and any test function $\varphi \in C^1_c([0,T] \times \RR^3)$, and the Young measure $\nu$ satisfies the integral identity
    \begin{align}\label{thm: Main Result Kinetic Eq}
        \int_0^T \int_\Omega \Bigl\langle \nu, 
        \partial_t\psi 
        +  \nabla \psi \cdot \uvec
        - \bigl(\xi \partial_\xi\psi - \psi \bigr) \div\uvec
        +\bigl( \xi \partial_\xi \psi - \psi \bigr) \osc 
        \Bigr\rangle
        \, \dd x \, \dd t
        =- \int_\Omega \Bigl\langle \nu^0, \psi(0,\cdot,\cdot) \Bigr\rangle\, \dd x
    \end{align}
    for any test function $\psi \in C^1_c([0,T)\times\Omega\times\RR)$, where $\osc$ is defined for almost all $(t,x,\xi) \in \Omega_T \times \RR$ as
    \begin{align}\label{Defi Osc}
        \osc(t,x,\xi) :=\frac{\oP(t,x) - \Peff(\xi) }{\lambda + 2\mu}.
    \end{align}
\end{theorem}
The integral identities $\eqref{thm: Main Result Momentum Eq}$, $\eqref{thm: Main Result Parabolic Eq}$ and $\eqref{thm: Main Result Kinetic Eq}$ in Theorem~\ref{thm: Main Result} correspond exactly to the weak formulation of the effective system $\eqref{Introduction: Effective Equation}$ subject to the boundary conditions $\eqref{NSK Parabolic Boundary Conditions}$ and the initial conditions $\eqref{Introduction Effective Equations Initial Conditions}$.
In \cite{HI_P} a corresponding homogenization result was obtained for the compressible Navier--Stokes equations with an isentropic pressure law under the restriction $\gamma \in (\frac{3}{2},\infty)$.
Our homogenization result Theorem~\ref{thm: Main Result} applies for the whole regime $\gamma \in (1,\infty)$ due to the additional quadratic contributions in the energy functional (cf.\ $\eqref{Introduction Energy})$ generalizing in a certain sense the results from \cite{HI_P}.
\par
Let us clarify in which sense our homogenization result Theorem~\ref{thm: Main Result} for the \system equations with highly oscillating initial density justifies the system $\eqref{NSK Parabolic Boundary Conditions}$, $\eqref{Introduction: Effective Equation}$, $\eqref{Introduction Effective Equations Initial Conditions}$ as an effective system for a compressible liquid-vapor flow in the regime where the number of phase boundaries is very large.
Agreeing on the fact, that the compressible liquid-vapor flow is correctly described on the detailed scale by the NSK equations with a pressure function of Van-der-Waals type, we approximate first the NSK equations on this scale by the \system equations, where the approximation takes place in the regime $\coup \to \infty$ and $\paracoup\to 0$.
Then, we derive an effective system for the \system system in the following way.
We take a sequence of initial data ${(\rho^0_n,\uvec^0_n,c^0_n)}_{n\in\NN}$, where we expect the initial density sequence to be highly oscillating between the liquid's and the vapor's density mimicking the large number of phase boundaries initially.
Considering a sequence of finite energy weak solutions ${(\rho_n,\uvec_n,c_n)}_{n\in\NN}$ to the \system equations $\eqref{NSK_parabolic}$--$\eqref{NSK Parabolic Boundary Conditions}$ with initial conditions ${(\rho^0_n,\uvec^0_n,c^0_n)}_{n\in\NN}$, we then anticipate that the effective system is found in the limit $n \to \infty$.
Theorem~\ref{thm: Main Result} tells us then, that, up to a subsequence, it is possible to perform this limit procedure leading to the effective system $\eqref{NSK Parabolic Boundary Conditions}$, $\eqref{Introduction: Effective Equation}$, $\eqref{Introduction Effective Equations Initial Conditions}$.
As the \system equations approximate the NSK equations in the limit $\coup \to \infty$ and $\paracoup\to 0$, it would we very interesting to perform this limit in the effective system.
The resulting system of equations could then be viewed as a homogenized system for the NSK equations with highly oscillating initial density.
In this context a natural question would then be whether the homogenization limit commutes with the approximation limit $\coup\to\infty$ and $\paracoup\to 0$.
However, performing this limit in the effective system, even just by formal arguments, is rather delicate due to the non-linearity that is present in the quadratic correction term in the pressure function $\Peff$.
More specifically, this non-linearity causes the issue that we have in general
\begin{equation*}
    \langle \nu , \xi^2\rangle \neq \langle \nu,\xi\rangle ^2 = \rho^2,
\end{equation*}
but to perform the formal limit $\coup\to\infty$, $\paracoup\to 0$ an equality in this relation is needed.\\
In order to make the effective system accessible for standard approximation methods, we rewrite the kinetic equation on $\nu$ as a kinetic equation for the corresponding cumulative distribution function.
We obtain the following result as a consequence of our main result Theorem~\ref{thm: Main Result}.

\begin{corollary}\label{Cor : Main Result}
    Let the hypotheses and notations of Theorem~\ref{thm: Main Result} hold true.
    Define for almost all $(t,x) \in \Omega_T$ the cumulative distribution function $f(t,x,\cdot)$ corresponding to $\nu_{(t,x)}$ via
    \begin{equation*}
        f(t,x,\xi) := \nu_{(t,x)}((-\infty,\xi]) 
    \end{equation*}
    for any $\xi \in \RR$.
    Then, for almost all $(t,x)$, we have that $\xi \mapsto f(t,x,\xi)$ is non-decreasing, right-continuous and satisfies
    \begin{equation*}
        f(t,x,\xi) = 0 \quad \forall \, \xi \in (-\infty,0), \qquad 
        \lim\limits_{\xi\to\infty}f(t,x,\xi) = 1.
    \end{equation*}
    The analogous statements hold true for the cumulative distribution function $f^0$ that is defined for almost all $x \in \Omega$ via
    \begin{equation*}
        f^0(x,\xi) := \nu^0((-\infty,\xi]) 
    \end{equation*}
    for any $\xi \in \RR$.
    Moreover, we have that
    \begin{align}\label{Cor : Main Result : Eq on f}
        \int_0^T \int_\Omega \int_\RR 
        \Bigl(
        f \partial_t \psi 
        +
        f\uvec\cdot \nabla \psi
        -
        \Bigl( 
        \xi f \div\uvec- \xi \mathcal{M}[f]
        \Bigr)
        \partial_\xi \psi 
        \Bigr)
        \, \dd \xi\, \dd x \, \dd t
        =
        -
        \int_\Omega \int_\RR
        f^0 \psi(0,\cdot,\cdot) \, \dd \xi \, \dd x
    \end{align}    
    holds for any test function $\psi \in C^1_c([0,T) \times \Omega\times\RR)$ with
    \begin{equation*}
        \mathcal{M}[f](t,x,\xi) := \frac{1}{\lambda + 2 \mu}\int_{[0,\xi]} \osc(t,x,\eta) \, \dd f(t,x,\eta),
    \end{equation*}
    for almost all $(t,x,\xi)\in \Omega_T \times \RR$ and we have
    \begin{equation*}
        \rho(t,x) = \int_{[0,\infty)} \xi \, \dd f(t,x,\xi),
    \quad 
    \oP(t,x) = \int_{[0,\infty)}  \Peff(\xi) \, \dd f(t,x,\xi)
    \end{equation*}
    for almost all $(t,x) \in \Omega_T$, where the integral with respect to $\dd f(t,x,\cdot)$ denotes the Lebesgue--Stieltjes integral with respect to $f(t,x,\cdot)$.
\end{corollary}

\section{Uniform Estimates}\label{Sec : Uniform Estimates}
Let the hypotheses of Theorem~\ref{thm: Main Result} hold true and let ${(\rho_n,\uvec_n,c_n)}_{n\in\NN}$ denote a sequence of finite energy weak solutions with initial data ${(\rho_n^0,\uvec^0_n,c^0_n)}_{n\in\NN}$ to the initial-boundary value problem of the \system equations $\eqref{NSK_parabolic}$--$\eqref{NSK Parabolic Boundary Conditions}$.
The goal of this section is to provide some \emph{a priori} estimates for the sequence of solutions ${(\rho_n,\uvec_n,c_n)}_{n\in\NN}$ that are uniform in $n$.
These \emph{a priori} estimates are crucial in order to exploit compactness arguments that enable us to pass to the limit $n \to \infty$ and to extract a limit system stated in Theorem~\ref{thm: Main Result}.
We start with the \emph{a priori} estimates that are provided by the energy dissipation (see $\eqref{Def: Renormalized Finite Energy Weak Solutions Energy Eq}$ in Definition~\ref{Def:Renormalized Finite Energy Weak Solutions}).
In order to ease the notation, we make the following convention throughout this whole section.
For positive coefficients $\Pgrowth,\mu,\lambda,\capi,\coup,\paracoup$, some positive time $T>0$ and given initial data $\rho^0 \in L^\tgamma(\Omega)$, $\uvec^0 \in L^{\frac{2\tgamma}{\tgamma-1}}(\Omega;\RR^3)$, $c^0 \in H^{1}(\Omega)$, where $\tgamma:=\max\{2,\gamma\}$, we denote by $\curlyC_0>0$ a generic constant that may vary from line to line, but only depends on
\begin{align*}
    \|\rho^0\|_{L^\tgamma(\Omega)}, \|\uvec^0\|_{L^{\frac{2\tgamma}{\tgamma-1}}(\Omega)}, \|c^0\|_{H^{1}(\Omega)},|\Omega|,\Pgrowth, \mu,\lambda, \capi, \coup, \paracoup,  T.
\end{align*}
\begin{lemma}\label{prop: Energy Dissipation Bounds}
    Let $T,\mu,\lambda,\capi,\coup,\paracoup>0$ and let $\Omega \subseteq \RR^3$ be a bounded domain with regular boundary.
    Assume that $P$ is admissible with growth rate $\gamma \in (1,\infty)$ and let $\tgamma := \max\{2,\gamma\}$.
    Let initial data $(\rho^0,\uvec^0,c^0)$ be given satisfying the regularity $\eqref{defi:renormalized:IC}$ and let $(\rho,\uvec,c)$ denote a finite energy weak solution of $\eqref{NSK_parabolic}$--$\eqref{NSK Parabolic Boundary Conditions}$ existing on $[0,T]$ and emanating from $(\rho^0,\uvec^0,c^0)$.
    Then we have the estimate
    \begin{align}\label{prop: Energy Dissipation Bounds Result1}
        &\|\sqrt{\rho}\uvec\|_{L^\infty(0,T;L^2(\Omega))}
        + \|\rho\|_{L^\infty(0,T;L^\gamma(\Omega))}
        + \|\rho-c\|_{L^\infty(0,T;L^2(\Omega))}
        + \|\nabla c\|_{L^\infty(0,T;L^2(\Omega))}
        \nonumber
        \\
        &\quad
        +\|\uvec\|_{L^2(0,T;H^{1}(\Omega))}
        + \|\partial_t c\|_{L^2(\OmegaT)}
        \leq \curlyC_0.
    \end{align}
\end{lemma}
\begin{proof}
    Since $(\rho,\uvec,c)$ is a finite energy weak solution, we have for almost all $t \in (0,T)$ the estimate
    \begin{align}\label{prop: Energy Dissipation Bounds Eq1}
        E(t) + \int_0^t \int_\Omega \bigl( \mu |\nabla \uvec|^2 + (\lambda + \mu) |\div\uvec|^2 + \beta |\partial_t c|^2\bigr) \, \dd x\, \dd \tau 
        \leq E_0
        \leq \curlyC_0,
    \end{align}
    where we have used $\eqref{bound pressure potential}$ and Hölder's inequality to obtain the second estimate and
    where $E(t)$ and $E_0$ are defined as in $\eqref{Introduction Energy}$ and $\eqref{Introduction Initial Energy}$, respectively.
    From $\eqref{prop: Energy Dissipation Bounds Eq1}$, we deduce by using $\eqref{bound pressure potential}$ and Hölder's and Poincaré's inequality
    \begin{align*}
        &\|\sqrt{\rho}\uvec\|_{L^\infty(0,T;L^2(\Omega))}
        +
        \|\rho\|_{L^\infty(0,T;L^\gamma(\Omega))}
        +
        \|\rho-c\|_{L^\infty(0,T;L^2(\Omega))}
        +
        \|\nabla c\|_{L^\infty(0,T;L^2(\Omega))}
        \\
        &\quad+
        \|\uvec\|_{L^2(0,T;H^1(\Omega))}
        +
        \|\partial_t c\|_{L^2(\OmegaT)}
        \leq \curlyC_0.
    \end{align*}
\end{proof}
Now, we obtain some improved \emph{a priori} estimates for $\rho$ and $c$. 
First, we control the mean value of $c$ and thereby the $L^\infty(0,T;L^\tgamma(\Omega))$-norm of $\rho$ by exploiting the parabolic equation $\eqref{Def: Renormalized Finite Energy Weak Solutions Parabolic Eq}$.
Then, we conclude an improved \emph{a priori} estimate for $c$ by exploiting standard parabolic regularity estimates.
\begin{lemma}\label{lem: Parabolic Regularity}
    Let the hypotheses and notations of Lemma~\ref{prop: Energy Dissipation Bounds} hold true.
    Then we have for any $t \in [0,T]$ that
    \begin{align}\label{lem: Parabolic Regularity I}
        \biggl|\int_\Omega c(t) \, \dd x \biggr| \leq \curlyC_0.
    \end{align}
    In particular, we have
    \begin{align}
        &\|\rho\|_{L^\infty(0,T;L^\tgamma(\Omega))}
        +
        \|\Peff(\rho)\|_{L^\infty(0,T;L^1(\Omega))}
        +
        \|\rho\uvec\|_{L^\infty(0,T;L^{\frac{2\tgamma}{\tgamma+1}}(\Omega))}
        +
        \|\rho\uvec\|_{L^2(0,T;L^{\frac{6\tgamma}{\tgamma+6}}(\Omega))}
        \nonumber
        \\
        &\quad+
        \|\rho\uvec\otimes\uvec\|_{L^2(0,T;L^\frac{3\tgamma}{\tgamma+3}(\Omega))}
        +
        \|c\|_{L^\infty(0,T;H^1(\Omega))}
        +
        \|c\|_{L^2(0,T;H^2(\Omega))}
        +
        \|\rho\nabla c\|_{L^2(0,T;L^{\frac{6\tgamma}{\tgamma+6}}(\Omega))}\nonumber
        \\
        &\leq 
        \curlyC_0.\label{lem: Parabolic Regularity II}
    \end{align}
\end{lemma}
\begin{proof}
    From $\eqref{Def: Renormalized Finite Energy Weak Solutions Parabolic Eq}$ with $\varphi \equiv 1$, we conclude for any $t \in [0,T]$ the relation
    \begin{align*}
        \int_\Omega c(t)\, \dd x
        =
        \int_\Omega c^0 \, \dd x
        +
        \frac{\coup}{\paracoup}\int_0^t\int_\Omega
        (\rho - c)\, \dd x \, \dd \tau
        =
        \int_\Omega
        c^0 
        \, \dd x
        +
        \frac{\coup}{\paracoup}\int_0^t\int_\Omega 
        (\rho^0-c)
        \, \dd x \, \dd \tau,
    \end{align*}
    where we have used the continuity equation in the last identity.
    This implies
    \begin{align*}
        \biggl|\int_\Omega c(t)\, \dd x\biggr|
        \leq 
        \|c^0\|_{L^1(\Omega)}
        +
        \frac{\coup T}{\paracoup}\|\rho^0\|_{L^1(\Omega)}
        +
        \frac{\coup}{\paracoup}
        \int_0^t\biggl|\int_\Omega c\, \dd x\biggr|\, \dd \tau.
    \end{align*}
    With Gronwall's inequality, we conclude $\eqref{lem: Parabolic Regularity I}$.
    Combining $\eqref{lem: Parabolic Regularity I}$, Poincaré's inequality and Lemma~\ref{prop: Energy Dissipation Bounds} yields
    \begin{align}\label{lem: Parabolic Regularity Pf I}
        \|c\|_{L^\infty(0,T;H^1(\Omega))}
        +
        \|\rho\|_{L^\infty(0,T;L^\tgamma(\Omega))}
        \leq \curlyC_0.
    \end{align}
    By $\eqref{bound for Peff}$, $\eqref{prop: Energy Dissipation Bounds Result1}$, Hölder's inequality and the Sobolev embedding $H^1(\Omega)\hookrightarrow L^6(\Omega)$ we conclude from $\eqref{lem: Parabolic Regularity Pf I}$ that
    \begin{align*}
        &\|\Peff(\rho)\|_{L^\infty(0,T;L^1(\Omega))}
        +
        \|\rho\uvec\|_{L^\infty(0,T;L^\frac{2\tgamma}{\tgamma+1}(\Omega))}
        +
        \|\rho\uvec\|_{L^2(0,T;L^{\frac{6\tgamma}{\tgamma+6}}(\Omega))}
        \\
        &\quad+
        \|\rho\uvec\otimes\uvec\|_{L^2(0,T;L^{\frac{3\tgamma}{\tgamma+3}}(\Omega))}
        \leq \curlyC_0.
    \end{align*}
    Since $c$ is the weak solution of the linear parabolic initial-boundary value problem
    \begin{alignat*}{2}
        \left\{
        \begin{aligned}
            \paracoup\partial_t c - \capi \Delta c + \coup c &= \coup \rho \quad &&\text{in } (0,T] \times \Omega,\\
            \nabla c \cdot \mathbf{n}_{\partial\Omega} &= 0 \quad &&\text{in } [0,T] \times \partial\Omega,\\
            c(\cdot,0) &= c^0 \quad &&\text{in } \Omega,
        \end{aligned}
        \right.
    \end{alignat*}
    standard parabolic regularity results (see e.g.~\cite[Chapter 7]{EPDE}) yield
    \begin{align*}
        \|c\|_{L^2(0,T;H^{2}(\Omega))}
        \leq
         \curlyC_0\Bigl( \|\rho\|_{L^2(0,T;L^2(\Omega))} + \|c^0\|_{H^{1}(\Omega)} \Bigr)
         \leq \curlyC_0,
    \end{align*}
    where we have used $\eqref{lem: Parabolic Regularity Pf I}$ to obtain the second inequality.
    From this inequality, we conclude by Hölder's inequality and the Sobolev embedding $H^1(\Omega)\hookrightarrow L^6(\Omega)$ that
    \begin{align*}
        \|\rho\nabla c\|_{L^2(0,T;L^\frac{6\tgamma}{\tgamma+6}(\Omega))}\leq \curlyC_0,
    \end{align*}
    which completes the proof of $\eqref{lem: Parabolic Regularity II}$.
\end{proof}
The improved \emph{a priori} estimate $\eqref{lem: Parabolic Regularity II}$ allows us to control the term $\Peff(\rho)$ by the initial data only in the space $L^\infty(0,T;L^1(\Omega))$.
This control is not appropriate for the limit procedure later on, as we can not guarantee that a sequence being uniformly bounded in this space is also equi-integrable (cf.\ Proposition~\ref{thm: Proof Main Result Part I}).
To overcome this, we derive some estimate that controls the term $\Peff(\rho)$ by the initial data in the space $L^\delta(\OmegaT)$ for some $\delta\in(1,\infty)$.
Note that a sequence being uniformly bounded in this space is in particular equi-integrable.
For the compressible Navier--Stokes equations with constant viscosity coefficients, such an estimate was derived by constructing an appropriate multiplier for the momentum equation with the Bogovskii operator (see \cite{FNP,FP}).
Let us recall at this point this operator including the properties that we will use in this work.

\begin{proposition}\label{Prop : Bogovskii Operator}
    Let $\Omega \subseteq \RR^3$ be a bounded Lipschitz domain and let $p,q,r \in (1,\infty)$.
    Then there exists a bounded linear operator
    \begin{equation*}
        \Bog \colon  \text{\L}^p(\Omega) := \biggl\{  \varphi \in L^p(\Omega) \mid \int_\Omega  \varphi \, \dd 
        x = 0\biggr\}
        \longrightarrow
        W^{1,p}_0(\Omega;\RR^3),
    \end{equation*}
    such that for any $f \in \text{\L}^p(\Omega)$, we have 
    \begin{equation}\label{Prop : Bogovskii Operator : Pointwise Relation}
        \div( \Bog [f]) = f \quad \text{a.e. in } \Omega  
    \end{equation}
    and
    \begin{equation}\label{Prop : Bogovskii Operator : W1p Estimate}
        \|\Bog[f]\|_{W^{1,p}_0(\Omega)} \leq c_1 \|f\|_{L^p(\Omega)}
    \end{equation}
    for some positive constant $c_1>0$ only depending on $p$ and $\Omega$.\\
    Furthermore, if $f \in \text{\L}^p(\Omega)$ satisfies
    \begin{equation*}
        f = \div \mathbf{g} \quad \text{in } \mathcal{D}^\prime(\Omega)
    \end{equation*}
    for some $\mathbf{g} \in L^r(\Omega;\RR^3)$ with $\mathbf{g}\cdot \mathbf{n}_{\partial\Omega} = 0$, then we have
    \begin{equation}\label{Prop : Bogovskii Operator : Estimate if Div}
        \|\Bog[f]\|_{L^r(\Omega)} \leq c_2 \|\mathbf{g}\|_{L^r(\Omega)}
    \end{equation}
    for some positive constant $c_2>0$ that only depends on $p,r$ and $\Omega$.\\
    In particular, for 
    \begin{equation*}
        q \in \begin{cases}
            [1,\frac{3p}{3-p}] &  \text{if } p \in [1,3),\\
            [1,\infty) & \text{if } p = 3,\\
            \{\infty\} & \text{if } p >3,
        \end{cases}
    \end{equation*}
    we have for any $f\in \text{\L}^p(\Omega)$ that
    \begin{equation*}
        \|\Bog[f]\|_{L^q(\Omega)} \leq c_3 \|f\|_{L^p(\Omega)}
    \end{equation*}
    for some positive constant $c_3>0$ that only depends on $p,q$ and $\Omega$.
\end{proposition}
\begin{proof}
    See \cite[Proposition~2.1]{FP}.
\end{proof}
In contrast to the compressible Navier--Stokes equation, we have an additional term $\coup\rho\nabla c$ present in the momentum equation, such that it is not clear whether the approach in \cite{FNP,FP} transfers to the \system equations.
However, by Lemma~\ref{lem: Parabolic Regularity}, we can control this additional term by the initial data in an appropriate way, such that the approach in \cite{FNP,FP} is applicable for the \system equations.
We obtain the following result for which we provide a complete proof in order to show that the approach of \cite{FNP,FP} applies for the \system equations.
\begin{lemma}\label{lem: Improved Space Time Pressure}
    Let the hypotheses and notations of Lemma~\ref{prop: Energy Dissipation Bounds} hold true.
    Then we have
    \begin{align}\label{lem: Improvec Space Time Pressure I}
        \int_0^T\int_\Omega \Peff(\rho) \rho^\theta\, \dd x \, \dd t
        \leq \curlyC_0, 
    \end{align}
    and in particular
    \begin{align}\label{lem: Improved Space Time Pressure II}
        \|\rho\|_{L^{\tgamma+\theta}(\OmegaT)}
        +
        \|\Peff(\rho)\|_{L^\delta(\OmegaT)}
        \leq \curlyC_0,
    \end{align}
    where 
    \begin{align*}
        \theta:= \min\biggl\{ \frac{2\tgamma-3}{3},1\biggr\}, \quad \delta:= \frac{\tgamma + \theta}{\tgamma}.
    \end{align*}
\end{lemma}
\begin{proof}
    Let $\phi \in C^\infty_c(B_1(0))$ denote a standard mollifier with
    \begin{align*}
        \int_{\RR^3} \phi(x) \, \dd x = 1, \quad 
        0 \leq \phi(x)\leq 1,\quad
        \phi(-x) = \phi(x) 
        \quad \forall \, x \in \RR^3.
    \end{align*}
    For $\delta>0$ and $x \in \RR^3$, we denote
    \begin{align*}
        \phi_\delta(x) := \frac{1}{\delta^3} \phi\Bigl(\frac{x}{\delta}\Bigr).
    \end{align*}
    We extend $\rho$ and $\uvec$ by zero onto $(0,T) \times \RR^3$. These extensions satisfy $\rho \in L^\infty(0,T;L^\tgamma(\RR^3))$ and $\uvec \in L^2(0,T;H^{1}_0(\RR^3;\RR^3))$.
    For some time dependent integrable function $f \in L^1((0,T)\times\RR^3)$ and $\delta>0$, we denote the spatial mollification of $f$ for almost all $t \in (0,T)$ via
    \begin{align*}
       S_\delta\bigl[f](t,x):= \int_{B_\delta(x)} \phi_\delta(x-y) f(t,y) \, \dd y \quad \forall \, x \in \Omega.
    \end{align*}
    Let us fix some function $b \in \mathcal{F}$ with the additional property that 
    \begin{align}\label{lem: Improved Space Time Pressure : Bound on b}
        |b(z)| + |b^\prime(z) z| \leq \curlyC_0 z^\theta \quad \forall \, z \in [0,\infty).
    \end{align}
    Note that the generic positive constant $\curlyC_0>0$ does not depend on $\delta>0$ by our convention.
    From the fact that $\rho$ satisfies the continuity equation in the renormalized sense in $\mathcal{D}^\prime ((0,T) \times \RR^3)$, we conclude that $S_\delta\bigl[ b(\rho) \bigr]$ satisfies its regularized version
    \begin{align*}
        \partial_tS_\delta\bigl[ b(\rho) \bigr] 
        =
        - \div \Bigl(  S_\delta\bigl[ b(\rho) \uvec \bigr]  \Bigr)
        - S_\delta \Bigl[ \bigl( b^\prime(\rho)\rho - b(\rho)\bigr) \div \uvec   \Bigr]
    \end{align*}
    almost everywhere on $(0,T) \times \RR^3$.
    In particular, we have that $\partial_t S_\delta\bigl[ b(\rho) \bigr] \in L^2(\OmegaT)$, at least.
    By introducing 
    \begin{align*}
        r_\delta:= \div\Bigl( S_\delta\bigl[ b(\rho) \bigr] \uvec \Bigr) - \div\Bigl( S_\delta\bigl[ b(\rho) \uvec \bigr] \Bigr)
    \end{align*}
    we can rewrite this relation as
    \begin{align}\label{lem: Improved Space Time Pressure : Mollified Renormalized Equation}
        \partial_tS_\delta\bigl[ b(\rho) \bigr] 
        =
        - \div \Bigl(  S_\delta\bigl[ b(\rho) \bigr] \uvec   \Bigr)
        - S_\delta \Bigl[ \bigl( b^\prime(\rho)\rho - b(\rho)\bigr) \div \uvec   \Bigr]
        +
        r_\delta.
    \end{align}
    Now, we have that $S_\delta\bigl[ b(\rho) \bigr]\uvec \cdot \mathbf{n}_{\partial\Omega} = 0$ in the sense of traces, which will be crucial in order to use the estimate $\eqref{Prop : Bogovskii Operator : Estimate if Div}$ from Proposition~\ref{Prop : Bogovskii Operator} later on.
    By Friedrichs' commutator lemma (see e.g.~\cite[Corollary 11.3]{FN}), we have for the additional term the strong convergence
    \begin{align}\label{lem: Improved Space Time Pressure : Convergence of r_delta}
        r_\delta \to 0 \quad \text{in } L^2(\OmegaT)
    \end{align}
    for $\delta \to 0$.\\
    We fix $\eta \in \mathcal{D}((0,T))$ with $0 \leq \eta \leq 1$ and introduce for $\delta>0$ a multiplier $\Psivec_\delta$ defined for almost all $t \in (0,T)$ via
    \begin{align*}
        \Psivec_\delta (t) := \eta(t) \Bog \Bigl( S_\delta \bigl[ b(\rho(t)) \bigr] - \fint_\Omega S_\delta\bigl[ b(\rho(t)) \bigr] \Bigr).
    \end{align*}
    As $S_\delta\bigl[ b(\rho) \bigr] \in L^\infty(0,T;L^\infty(\Omega))$, we infer by estimate $\eqref{Prop : Bogovskii Operator : W1p Estimate}$ from Proposition~\ref{Prop : Bogovskii Operator} that
    \begin{align}\label{lem: Improved Space Time Pressure : Regularity Psi}
        \Psivec_\delta \in L^\infty(0,T;W^{1,p}_0(\Omega;\RR^3))
    \end{align}
    for any fixed $\delta>0$ and for any $p \in (1,\infty)$.
    With relation $\eqref{Prop : Bogovskii Operator : Pointwise Relation}$ from Proposition~\ref{Prop : Bogovskii Operator} we obtain that
    \begin{align}\label{lem: Improved Space Time Pressure : Regularity div Psi}
        \div \Psivec_\delta = \eta S_\delta\bigl[ b(\rho) \bigr] - \eta\fint_\Omega S_\delta\bigl[ b(\rho) \bigr] \in L^\infty(0,T;L^\infty(\Omega)).
    \end{align}
    Using equation $\eqref{lem: Improved Space Time Pressure : Mollified Renormalized Equation}$, we infer that
    \begin{align*}
        \partial_t \Psivec_\delta 
        &=
        - \Bog\biggl( \div\Bigl( S_\delta\bigl[ b(\rho) \bigr] \uvec \Bigr) \biggr)
        - \Bog\biggl( S_\delta\Bigl[ \bigl(b^\prime(\rho)\rho - b(\rho) \bigr) \div\uvec\Bigr] - \fint_\Omega S_\delta\Bigl[ \bigl( b^\prime(\rho)\rho - b(\rho)\bigr) \div\uvec \Bigr]\biggr)\\
        &\qquad
        +
        \Bog\biggl( r_\delta - \fint_\Omega r_\delta  \biggr)
        +
        \partial_t \eta \, \Bog\Bigl( S_\delta[b(\rho)] - \fint_\Omega S_\delta [b(\rho)]\Bigr).
    \end{align*}
    In particular, we have
    \begin{align}\label{lem: Improved Space Time Pressure : Regularity partial_t Psi}
        \partial_t \Psivec_\delta \in L^2(0,T;W^{1,2}_0(\Omega;\RR^3)).
    \end{align}
    Using the fact that $\Psivec_\delta$ satisfies the regularity $\eqref{lem: Improved Space Time Pressure : Regularity Psi}, \eqref{lem: Improved Space Time Pressure : Regularity div Psi}, \eqref{lem: Improved Space Time Pressure : Regularity partial_t Psi}$ in combination with the fact that $\eta$ has compact support in $(0,T)$, we conclude by a density argument that $\Psivec$ is an admissible test function for the momentum equation $\eqref{Def: Renormalized Finite Energy Weak Solutions Momentum Eq}$.
    By using $\Psivec_\delta$ as a test function and exploiting the relation $\eqref{Prop : Bogovskii Operator : Pointwise Relation}$ from Proposition~\ref{Prop : Bogovskii Operator}, we obtain for any fixed $\delta>0$ the relation
    \begin{align}\label{lem: Improved Space Time Pressure : Momentum Equation Tested}
        \int_0^T \int_\Omega \eta \,\Peff(\rho) S_\delta\bigl[b(\rho)\bigr] \, \dd x \, \dd t = \sum\limits_{i=1}^9 I_i^\delta
    \end{align}
    with
    \begingroup
    \allowdisplaybreaks[4]
    \begin{align*}
        I_1^\delta  &:= \int_0^T \int_\Omega \eta \, \Peff(\rho) \fint_\Omega S_\delta\bigl[ b(\rho) \bigr] \, \dd x \, \dd t,\\
        I_2^\delta &:= \int_0^T \int_\Omega \eta \,\rho\uvec \cdot \Bog\biggl(  \div\Bigl( S_\delta\bigl[b(\rho)\bigr] \uvec \Bigr)\biggr)\, \dd x \, \dd t, \\
        I_3^\delta &:= \int_0^T \int_\Omega \eta \, \rho \uvec \cdot \Biggl(  \Bog\biggl( S_\delta\Bigl[ (b^\prime(\rho)\rho - b(\rho)) \div\uvec \Bigr]   - \fint_\Omega S_{\delta}\Bigl[ \bigl(b^\prime(\rho)\rho - b(\rho)\bigr) \div\uvec \Bigr] \biggr)  \Biggr) \, \dd x\, \dd t,\\
        I_4^\delta &:= -\int_0^T \int_\Omega \eta \, \rho \uvec \cdot \Bog\biggl( r_\delta - \fint_\Omega r_\delta \biggr)\, \dd x \, \dd t,\\
        I_5^\delta &:= \mu \int_0^T \int_\Omega \eta \,\nabla \uvec : \nabla \Bog\biggl( S_\delta\bigl[b(\rho)\bigr] - \fint_\Omega S_\delta\bigl[b(\rho)\bigr] \biggr)\, \dd x \, \dd t,\\
        I_6^\delta &:= - \int_0^T \int_\Omega \eta \,\rho\uvec\otimes\uvec : \nabla \Bog\biggl( S_\delta\bigl[b(\rho)\bigr] - \fint_\Omega S_\delta\bigl[b(\rho)\bigr] \biggr) \, \dd x \, \dd t,\\ 
        I_7^\delta &:= (\lambda + \mu ) \int_0^T \int_\Omega \eta \, \div\uvec \Bigl( S_\delta\bigl[b(\rho)\bigr] - \fint_\Omega S_\delta\bigl[b(\rho)\bigr]\Bigr) \, \dd x \, \dd t,\\
        I_8^\delta &:= - \coup \int_0^T \int_\Omega \eta \, \rho \nabla c \cdot \Bog\biggl( S_\delta\bigl[b(\rho)\bigr] - \fint_\Omega S_\delta\bigl[b(\rho)\bigr] \biggr)\, \dd x \, \dd t,\\
        I_9^\delta &:= - \int_0^T \int_\Omega \partial_t \eta \, \rho \uvec \cdot \Bog\biggl( S_\delta\bigl[b(\rho)\bigr] - \fint_\Omega S_\delta\bigl[b(\rho)\bigr] \biggr)\, \dd x \, \dd t.
    \end{align*}
    \endgroup
    Our goal is to estimate the right hand side of equation $\eqref{lem: Improved Space Time Pressure : Momentum Equation Tested}$ independent of $\delta$.
    In order to do so, we will use the estimate $\eqref{lem: Improved Space Time Pressure : Bound on b}$.
    More specifically, by applying standard results on the convolution of a function (see e.g.~\cite{EPDE}), we deduce from $\eqref{lem: Improved Space Time Pressure : Bound on b}$ that
    \begin{align}
        &\|S_\delta[b(\rho)]\|_{L^p(0,T;L^q(\Omega))}
        +
        \|S_\delta[b^\prime(\rho)\rho]\|_{L^p(0,T;L^q(\Omega))}\nonumber
        \\
        &\leq \curlyC_0\Bigl( \|b(\rho)\|_{L^p(0,T;L^q(\Omega))} + \|b^\prime(\rho)\rho\|_{L^p(0,T;L^q(\Omega))} \Bigr)\nonumber
        \\
        &\leq \curlyC_0 \|\rho^\theta\|_{L^p(0,T;L^q(\Omega))}
        \qquad \forall\, p,q\in[1,\infty].\label{mollification estimate}
    \end{align}
    With Lemma~\ref{lem: Parabolic Regularity}, Proposition~\ref{Prop : Bogovskii Operator} and $\eqref{mollification estimate}$, we estimate the terms $I_i^\delta$ for $i \in \{1,2,\cdots,9\}$ as follows.\\
    For $I_1^\delta$, we have
    \begin{align*}
        |I_1^\delta| 
        &\leq 
        \|\eta\|_{L^\infty(0,T)}
        \|\fint_\Omega S_\delta\bigl[ b(\rho) \bigr] \|_{L^\infty((0,T))} 
        \|\Peff(\rho)\|_{L^1(\OmegaT)}
        \\
        &\leq
        \curlyC_0
        \|\eta\|_{L^\infty(0,T)}
        \|\rho^\theta\|_{L^\infty(0,T;L^1(\Omega))} \|\Peff(\rho)\|_{L^1(\OmegaT)}
        \leq 
        \curlyC_0\|\eta\|_{L^\infty(0,T)},
    \end{align*}
    since $\theta\leq \tgamma$.\\
    For $I_\delta^2$, we set $r_1:=\frac{6\tgamma}{5\tgamma-6}$ and $r_2:=\frac{3\tgamma}{2\tgamma-3}$ and estimate
    \begin{align*}
        |I_2^\delta|
        &\leq 
        \|\eta\|_{L^\infty(0,T)}
        \|\rho \uvec\|_{L^2(0,T;L^{\frac{6\tgamma}{\tgamma+6}}(\Omega))}  \Bigl\|\Bog\Bigl(\div\bigl(S_\delta[b(\rho)] \uvec\bigr) \Bigr) \Bigr\|_{L^2(0,T;L^{r_1}(\Omega))}
        \\
        &\leq 
        \curlyC_0 \|\eta\|_{L^\infty(0,T)} \|S_\delta[b(\rho)]\uvec\|_{L^2(0,T;L^{r_1}(\Omega))}
        \\
        &\leq 
        \curlyC_0 \|\eta\|_{L^\infty(0,T)}  \|S_\delta[b(\rho)]\|_{L^\infty(0,T;L^{r_2}(\Omega))}\|\uvec\|_{L^2(0,T;L^6(\Omega))}
        \\
        &\leq 
        \curlyC_0 \|\eta\|_{L^\infty(0,T)} \|\rho^\theta \|_{L^\infty(0,T;L^{r_2}(\Omega))}
        \leq \curlyC_0 \|\eta\|_{L^\infty(0,T)},
    \end{align*}
    since $\theta r_2 \leq \tgamma$.\\
    For $I_\delta^3$, we set $r_3:=\max\Bigl\{\frac{6\tgamma}{7\tgamma-6},1\Bigr\}$ and estimate
    \begin{align*}
        |I_3^\delta|
        &\leq 
        \curlyC_0 \|\eta\|_{L^\infty(0,T)} \Bigl\| \Bog \Bigl( S_\delta\bigl[(b^\prime(\rho)\rho - b(\rho))\div\uvec\bigr] - \fint_\Omega S_\delta \bigl[ (b^\prime(\rho)\rho - b(\rho))\div\uvec \bigr]\Bigr) \Bigr\|_{L^2(0,T;L^{r_1}(\Omega))}
        \\
        &\leq 
        \curlyC_0 \|\eta\|_{L^\infty(0,T)}\Bigl\|S_\delta \bigl[(b^\prime(\rho)\rho-b(\rho))\div\uvec \bigr]\Bigr\|_{L^2(0,T;L^{r_3}(\Omega))}
        \\
        &\leq 
        \curlyC_0 \|\eta\|_{L^\infty(0,T)} \|(b^\prime(\rho)\rho -b(\rho))\div\uvec\|_{L^2(0,T;L^{r_3}(\Omega))}
        \\
        &\leq 
        \curlyC_0 \|\eta\|_{L^\infty(0,T)}\|b^\prime(\rho)\rho -b(\rho)\|_{L^\infty(0,T;L^{r_2}(\Omega))} \|\uvec\|_{L^2(0,T;H^1(\Omega))}
        \\
        &\leq \curlyC_0 \|\eta\|_{L^\infty(0,T)} \|\rho^\theta\|_{L^\infty(0,T;L^{r_2}(\Omega))}
        \leq 
        \curlyC_0 \|\eta\|_{L^\infty(0,T)},
    \end{align*}
    since $\theta r_2 \leq \tgamma$.\\
    For $I_4^\delta$, we have due to $\eqref{lem: Improved Space Time Pressure : Convergence of r_delta}$ that
    \begin{align*}
        |I_4^\delta|
        &\leq 
        \curlyC_0 \|\eta\|_{L^\infty(0,T)}\Bigl\| \Bog\Bigl( r_\delta - \fint_\Omega r_\delta \Bigr) \Bigr\|_{L^2(0,T;L^{r_1}(\Omega))}
        \leq 
        \curlyC_0 \|\eta\|_{L^\infty(0,T)}\|r_\delta \|_{L^2(0,T;L^{r_3}(\Omega))} \to 0
    \end{align*}
    as $\delta \to 0$, since $r_3 \leq 2$.\\
    For $I_5^\delta$, we obtain 
    \begin{align*}
        |I_5^\delta|
        &\leq 
        \curlyC_0\|\eta\|_{L^\infty(0,T)} \|\nabla \uvec\|_{L^2(\OmegaT)}\Bigl\| \nabla\Bog\Bigl(S_\delta[b(\rho)] - \fint_\Omega S_\delta[b(\rho)] \Bigr) \Bigr\|_{L^2(\OmegaT)}
        \\
        &\leq 
        \curlyC_0 \|\eta\|_{L^\infty(0,T)} \|S_\delta[b(\rho)]\|_{L^2(\OmegaT)}
        \leq 
        \curlyC_0 \|\eta\|_{L^\infty(0,T)} \|\rho^\theta\|_{L^\infty(0,T;L^2(\Omega))}
        \\
        &\leq 
        \curlyC_0 \|\eta\|_{L^\infty(0,T)},
    \end{align*}
    since $2\theta \leq 2\leq \tgamma$.\\
    For $I_6^\delta$, we estimate
    \begin{align*}
        |I_6^\delta|
        &\leq 
        \|\eta\|_{L^\infty(0,T)}
        \|\rho\uvec\otimes\uvec\|_{L^2(0,T;L^{\frac{3\tgamma}{\tgamma+3}}(\Omega))} \Bigl\|\nabla \Bog\Bigl( S_\delta[b(\rho)] - \fint_\Omega S_\delta[b(\rho)]\Bigr) \Bigr\|_{L^2(0,T;L^{r_2}(\Omega))}
        \\
        &\leq 
        \curlyC_0 \|\eta\|_{L^\infty(0,T)} \|S_\delta[b(\rho)]\|_{L^2(0,T;L^{r_2}(\Omega))}
        \leq 
        \curlyC_0 \|\eta\|_{L^\infty(0,T)} \|\rho^\theta\|_{L^2(0,T;L^{r_2}(\Omega))} 
        \\
        &\leq 
        \curlyC_0 \|\eta\|_{L^\infty(0,T)},
    \end{align*}
    since $\theta r_2 \leq \tgamma$.\\
    For $I_7^\delta$, we have 
    \begin{align*}
         |I_7^\delta|
         &\leq 
         \curlyC_0 \|\eta\|_{L^\infty(0,T)} \|\div\uvec\|_{L^2(\OmegaT)} \|S_\delta[b(\rho)]\|_{L^2(\OmegaT)}
         \\
         &\leq 
         \curlyC_0 \|\eta\|_{L^\infty(0,T)} \|\rho^\theta\|_{L^\infty(0,T;L^2(\Omega)}
         \leq 
         \curlyC_0 \|\eta\|_{L^\infty(0,T)},
    \end{align*}
    since $2 \theta \leq 2 \leq \tgamma$.\\
    For $I_8^\delta$, we obtain
    \begin{align*}
        |I_8^\delta|
        &\leq 
        \|\eta\|_{L^\infty(0,T)}
        \|\rho\nabla c\|_{L^2(0,T;L^{\frac{6\tgamma}{\tgamma + 6}}(\Omega))}
        \Bigl\|\Bog\Bigl( S_\delta[b(\rho)] - \fint_\Omega S_\delta[b(\rho)] \Bigr)\Bigr\|_{L^2(0,T;L^{r_1}(\Omega))}
        \\
        &\leq
        \curlyC_0 \|\eta\|_{L^\infty(0,T)} \|S_\delta[b(\rho)]\|_{L^2(0,T;L^{r_3}(\Omega))}
        \leq
        \curlyC_0 \|\eta\|_{L^\infty(0,T)} \|b(\rho)\|_{L^2(0,T;L^{r_3}(\Omega))}
        \\
        &\leq 
        \curlyC_0\|\eta\|_{L^\infty(0,T)}\|\rho^\theta\|_{L^\infty(0,T;L^{r_3}(\Omega))}
        \leq \curlyC_0 \|\eta\|_{L^\infty(0,T)},
    \end{align*}
    since $\theta r_3 \leq 2 \leq  \tgamma$.\\
    For $I_9^\delta$, we set $r_4:=\frac{6\tgamma}{5\tgamma-3}$ and estimate
    \begin{align*}
        |I_9^\delta|
        &\leq 
        \|\partial_t\eta\|_{L^1(0,T)}
        \|\rho\uvec\|_{L^\infty(0,T;L^{\frac{2\tgamma}{\tgamma+1}}(\Omega))}\Bigl\|\Bog\bigl( S_\delta[b(\rho)] - \fint_\Omega S_\delta[b(\rho)]\Bigr) \Bigr\|_{L^\infty(0,T;L^{\frac{2\tgamma}{\tgamma-1}}(\Omega))}
        \\
        &\leq
        \curlyC_0
        \|\partial_t \eta \|_{L^1(0,T)}
        \|S_\delta[b(\rho)]\|_{L^\infty(0,T;L^{r_4}(\Omega))}
        \leq 
        \curlyC_0\|\partial_t \eta \|_{L^1(0,T)} \|\rho^\theta\|_{L^\infty(0,T;L^{r_4}(\Omega))}
        \\
        &\leq \curlyC_0 \|\partial_t \eta \|_{L^1(0,T)},
    \end{align*}
    since $\theta r_4 \leq 2 \leq  \tgamma$.\\
    In total we deduce after passing to the limit $\delta\to 0$ in $\eqref{lem: Improved Space Time Pressure : Momentum Equation Tested}$ with Lebesgue's dominated convergence theorem that
    \begin{align}\label{lem: Improved Space Time Pressure : Final Estimate dependent on b}
        \int_0^T \int_\Omega 
        \eta \, \Peff(\rho) b(\rho) \, \dd x \, \dd t \leq \curlyC_0\Bigl(\|\eta\|_{L^\infty(0,T)} + \|\partial_t \eta \|_{L^1(0,T)}\Bigr)
    \end{align}
    for any $\eta \in \mathcal{D}((0,T))$ and any $b \in \mathcal{F}$ satisfying $\eqref{lem: Improved Space Time Pressure : Bound on b}$.
    Now, we choose for $k \in \NN$ non-negative functions $b_k \in \mathcal{F}$ with the property that
    \begin{align}\label{lem: Improved Space Time Pressure : pointwise convergence approximating b_k}
        0 \leq b_k(z) \to z^\theta \quad \forall \, z \in [0,\infty) 
    \end{align}
    and
    \begin{align*}
        |b_k(z)| + |b_k^\prime(z) z | \leq \curlyC_0 z^\theta \quad \forall z\, \in [0,\infty),
    \end{align*}
    where the generic positive constant $\curlyC_0>0$ does not depend on $k$ by our convention.
    With $\eqref{lem: Improved Space Time Pressure : Final Estimate dependent on b}$, $\eqref{lem: Improved Space Time Pressure : pointwise convergence approximating b_k}$ and Fatou's lemma, we obtain
    \begin{align}\label{estimate with eta}
        \int_0^T \int_\Omega 
        \eta \,
        \Peff(\rho) \rho^\theta 
        \, \dd x \, \dd t
        \leq 
        \liminf\limits_{k\to\infty}
        \int_0^T \int_\Omega 
        \Peff(\rho) b_k(\rho)
        \, \dd x \, \dd t
        \leq 
        \curlyC_0 \Bigl( \|\eta\|_{L^\infty(0,T)} + \|\partial_t \eta\|_{L^1(0,T)} \Bigr),
    \end{align}
    for any $\eta \in \mathcal{D}(0,T)$.
    By approximating $\mathrm{Id}_{(0,T)}$ with functions $\eta\in \mathcal{D}((0,T))$, we conclude from $\eqref{estimate with eta}$ precisely $\eqref{lem: Improvec Space Time Pressure I}$.
    From $\eqref{lem: Improvec Space Time Pressure I}$ we deduce $\eqref{lem: Improved Space Time Pressure II}$ by using $\eqref{bound for Peff}$ and Hölder's inequality.
\end{proof}
As a consequence of the estimates in Lemma~\ref{prop: Energy Dissipation Bounds}, Lemma~\ref{lem: Parabolic Regularity} and Lemma~\ref{lem: Improved Space Time Pressure}, we obtain the following uniform bounds for a sequence of finite energy weak solutions ${(\rho_n,\uvec_n,c_n)}_{n\in\NN}$ to the initial-boundary value problem $\eqref{NSK_parabolic}$--$\eqref{NSK Parabolic Boundary Conditions}$ emanating from initial conditions ${(\rho^0_n,\uvec^0_n,c^0_n)}_{n\in\NN}$ satisfying the uniform bound $\eqref{thm: Main Result Result 1}$.

\begin{proposition}\label{Prop: Uniform-in-n Bounds}
    Let $T,\mu,\lambda,\capi,\coup,\paracoup>0$ and let $\Omega \subseteq \RR^3$ be a bounded domain with regular boundary.
    Assume that $P$ is admissible with growth rate $\gamma \in (1,\infty)$, $\tgamma:=\max\{2,\gamma\}$ and let $\theta,\delta$ be defined as in Lemma~\ref{lem: Improved Space Time Pressure}.
    For $n \in \NN$, let initial data
    \begin{align*}
        \rho_n^0 \in L^\tgamma(\Omega), \quad 
        \rho_n^0 \geq 0 \quad \text{a.e.}, \quad
        \uvec_n^0 \in L^{\frac{2\tgamma}{\tgamma-1}}(\Omega;\RR^3), \quad
        c^0_n \in H^{1}(\Omega)
    \end{align*}
    be given satisfying the uniform bound $\eqref{thm: Main Result Result 1}$ and let $(\rho_n,\uvec_n,c_n)$ denote a finite energy weak solutions to the initial-boundary value problem $\eqref{NSK_parabolic}$--$\eqref{NSK Parabolic Boundary Conditions}$ existing on $[0,T]$ and emanating from the initial data $(\rho^0_n,\uvec^0_n,c^0_n)$.\\
    Then we have that
    \begin{equation}
        \begin{aligned}
            &{(\rho_n)}_{n\in\NN} \quad \text{is uniformly bounded in } L^\infty(0,T;L^\tgamma(\Omega))\cap L^{\tgamma+\theta}(\OmegaT),\\
            &{(\rho_n\uvec_n)}_{n\in\NN} \quad \text{is uniformly bounded in } L^\infty(0,T;L^\frac{2\tgamma}{\tgamma+1}(\Omega;\RR^3)),\\
            &{(\rho_n\uvec_n\otimes\uvec_n)}_{n\in\NN} \quad \text{is uniformly bounded in } L^2(0,T;L^{\frac{3\tgamma}{\tgamma+3}}(\Omega;\RR^{3\times3})),\\
            &{(\Peff(\rho_n))}_{n\in\NN} \quad \text{is uniformly bounded in } L^\delta(\OmegaT),\\
            &{(c_n)}_{n\in\NN} \quad \text{is uniformly bounded in } L^2(0,T;H^2(\Omega))\cap L^\infty((0,T);H^1(\Omega)),\\
            &{(\rho_n\nabla c_n)}_{n\in\NN} \quad \text{is uniformly bounded in } L^2(0,T;L^\frac{6\tgamma}{\tgamma+6}(\Omega;\RR^3)),\\
            &{(\partial_t c_n)}_{n\in\NN} \quad \text{is uniformly bounded in } L^2(\OmegaT).
        \end{aligned}\label{Prop: Uniform-in-n Bounds Result 1}
    \end{equation}
    Moreover, we have that
    \begin{equation*}
        \begin{aligned}
            &{(\rho_n)}_{n\in\NN} \quad \text{is uniformly continuous in } W^{-1,\frac{2\tgamma}{\tgamma+1}}(\Omega),\\
            &{(\rho_n \uvec_n)}_{n\in\NN} \quad \text{is uniformly continuous in } W^{-1,q}(\Omega;\RR^3),
        \end{aligned}
    \end{equation*}
    where
    \begin{equation*}
        q := \min\Bigl( \frac{3\tgamma}{\tgamma+3}, \delta, \frac{6\tgamma}{\tgamma+6}, 2 \Bigr) \in (1,2].
    \end{equation*}
\end{proposition}
\begin{proof}
    The uniform bounds $\eqref{Prop: Uniform-in-n Bounds Result 1}$ follow from Lemma~\ref{prop: Energy Dissipation Bounds}, Lemma~\ref{lem: Parabolic Regularity} and Lemma~\ref{lem: Improved Space Time Pressure}.
    To see that ${(\rho_n)}_{n\in\NN}$ is uniformly continuous in $W^{-1,\frac{2\tgamma}{\tgamma+1}}(\Omega)$, we fix some $\phi \in \mathcal{D}(\Omega)$.
    For any $n \in \NN$, the continuity equation $\eqref{Def: Renormalized Finite Energy Weak Solutions Continuity Eq}$ holds and thus
    \begin{equation*}
        \frac{\dd}{\dd t}\int_\Omega \rho_n \phi \, \dd x 
        = \int_\Omega  \rho_n\uvec_n \cdot \nabla \phi \,\dd x
        \quad \text{in } \mathcal{D}^\prime((0,T)).
    \end{equation*}
    This relation implies for any $0\leq s \leq t \leq T$ and any $n \in \NN$ that
    \begin{align*}
        &\int_\Omega (\rho_n(t) - \rho_n(s))\phi \, \dd x = 
        \int_s^t \int_\Omega \rho_n\uvec_n(\tau) \cdot \nabla \phi \, \dd x \dd \tau
        \\
        &\leq
        \|\rho_n\uvec_n\|_{L^\infty(0,T;L^{\frac{2\tgamma}{\tgamma+1}}(\Omega))} \|\phi\|_{W^{1,\frac{2\tgamma}{\tgamma-1}}_0(\Omega)}|t-s|
    \end{align*}
    and with the uniform bounds in $\eqref{Prop: Uniform-in-n Bounds Result 1}$ we conclude that $(\rho_n)_{n\in\NN}$ is uniformly continuous in $W^{-1,\frac{2\tgamma}{\tgamma+1}}(\Omega)$.
    To see that ${(\rho_n\uvec_n)}_{n\in\NN}$ is uniformly continuous in $W^{-1,q}(\Omega;\RR^3)$, we proceed analogously and fix some $\Psivec\in \curlyD(\Omega;\RR^3)$.
    For any $n \in \NN$, the momentum equation $\eqref{Def: Renormalized Finite Energy Weak Solutions Momentum Eq}$ holds and thus
    \begin{equation*}
        \begin{aligned}
            \frac{\dd}{\dd t} \int_\Omega (\rho\uvec)\cdot \Psivec \, \dd x
            &=
            \int_\Omega 
            \bigl(
            \rho_n \uvec_n \otimes \uvec_n : \nabla \Psivec
            + 
            \Peff(\rho_n) \div \Psivec
            \bigr)
            \, \dd x \, \dd t
            \\
            &\quad - \int_\Omega 
            \bigr(
            \mu \nabla \uvec_n : \nabla \Psivec
            +
            (\lambda+\mu)\div\uvec_n\div\Psivec
            -
            \coup \rho_n \nabla c_n  \cdot \Psivec
            \bigr)
            \, \dd x
            \quad \text{in } \mathcal{D}^\prime((0,T)).
        \end{aligned}
    \end{equation*}
    This relation implies for any $0\leq s \leq t \leq T$ and any $n \in \NN$ that
    \begin{equation*}
        \begin{aligned}
            &\int_\Omega \Bigl((\rho_n \uvec_n) (t) - (\rho_n\uvec_n)(s) \Bigr) \cdot \Psivec \, \dd x \\
            &\leq
            \curlyC_0
            \int_s^t \Bigl( 
            \|\rho_n\uvec_n\otimes\uvec_n\|_{L^{\frac{3\tgamma}{\tgamma+3}}(\Omega)}
            +
            \|\Peff(\rho_n)\|_{L^\delta(\Omega)}
            +
            \|\nabla \uvec_n\|_{L^2(\Omega)}
            +
            \|\rho_n\nabla c_n\|_{L^{\frac{6\tgamma}{\tgamma+6}}(\Omega)} 
            \Bigr) \, \dd t\\
            &\qquad \times \|\Psivec\|_{W^{1,q^\prime}_0(\Omega)}
        \end{aligned}
    \end{equation*}
    where $q^\prime \in (1,\infty)$ denotes the conjugate Hölder exponent to $q$, i.e., $\frac{1}{q} + \frac{1}{q^\prime} = 1$. 
    By the uniform bounds in $\eqref{Prop: Uniform-in-n Bounds Result 1}$, we conclude from this estimate that ${(\rho_n\uvec_n)}_{n\in\NN}$ is uniformly continuous in $W^{-1,q}(\Omega;\RR^3)$.
\end{proof}

\section{Proof of the Main Theorem}\label{Sec : Proof of the Main Theorem}

This section is devoted to the complete proof for our main result Theorem~\ref{thm: Main Result}.
Before starting with the proof, let us recall the following result on Young measures that we will use in the sequel.
For a more detailed exposition on Young measures we refer to \cite{PRG}.

\begin{proposition}\label{Prop : Young Measures}
    Let $d \in \NN$ and let $Q\subseteq \RR^d$ be a bounded domain.
    For $n \in \NN$, let integrable functions $f_n \in L^1(Q)$ be given such that ${(f_n)}_{n\in\NN}$ is uniformly bounded in $L^1(Q)$.
    Then, there exists some Young measure $\nu \in \mathcal{Y}(Q;\RR)$ such that, after passing to a non-relabeled subsequence,
    \begin{equation*}
        \nu_{f_n} \weakstar \nu \quad \text{in } L^\infty_{\mathrm{w}^*}(Q;\mathcal{M}(\RR)).
    \end{equation*}
    Moreover, for any $b \in C^0(\RR)$ such that ${(b(f_n))}_{n\in\NN}$ is uniformly integrable, i.e.
    \begin{equation*}
        \lim\limits_{k\to \infty} \sup\limits_{n\in\NN} \int_{\{|b(f_n)|\geq k\}}
        |b(f_n)|\, \dd x = 0,
    \end{equation*}
    we have that $\langle \nu , b \rangle \in L^1(Q)$ and
    \begin{equation*}
        b(f_n) \weak \langle \nu,b\rangle \quad \text{in } L^1(Q).
    \end{equation*}
    In particular, for $p \in (1,\infty)$, we have for any $b \in C^0(\RR)$ such that ${(b(f_n))}_{n\in\NN}$ is uniformly bounded in $L^p(Q)$ that $\langle \nu,b\rangle \in L^p(Q)$ holds and that
    \begin{equation*}
        b(f_n) \weak \langle \nu ,b\rangle \quad \text{in } L^p(Q).
    \end{equation*}
\end{proposition}
\begin{proof}
    See the (more general) Theorem 6.2 in \cite{PRG}. 
\end{proof}
We emphasize that Proposition~\ref{Prop : Young Measures} is crucial to obtain a closed effective model in the limit $n \to \infty$ since, under the hypotheses of Theorem~\ref{thm: Main Result}, the sequence of densities ${(\rho_n)}_{n\in\NN}$ is expected to be highly oscillating and the pressure function $\Peff$ is not an affine function.\\
Let us now start with the proof of our main result Theorem~\ref{thm: Main Result}.
We split this proof into two parts: In a first part, we show that the uniform bounds established in Section~\ref{Sec : Uniform Estimates} (cf.\ Proposition~\ref{Prop: Uniform-in-n Bounds}) are strong enough in order to extract some limit quantities in the form of a Young measure and some deterministic hydrodynamic quantities that satisfy $\eqref{thm: Main Result Continuity Eq}, \eqref{thm: Main Result Momentum Eq}$ and $\eqref{thm: Main Result Parabolic Eq}$ including a relation between the hydrodynamic part and the Young measure.
In a second part, we then close the model by verifying that the Young measure satisfies the kinetic equation $\eqref{thm: Main Result Kinetic Eq}$ resulting into the proof of Theorem~\ref{thm: Main Result}.
The first part of the proof of Theorem~\ref{thm: Main Result} is given by the following proposition.
\begin{proposition}\label{thm: Proof Main Result Part I}
    Let the hypotheses and notations of Theorem~\ref{thm: Main Result} hold true and let $\theta,\delta$ be defined as in Lemma~\ref{lem: Improved Space Time Pressure}.
    Then, there exist a Young measure $\nu \in \mathcal{Y}(\Omega_T;\RR)$ and functions $\uvec \in L^2(0,T;H^{1}_0(\Omega;\RR^3))$, $c \in L^2(0,T;H^{2}(\Omega))\cap C([0,T];H^{1}(\Omega))$, $\rho \in C_{\mathrm{w}}([0,T];L^\tgamma(\Omega))\cap L^{\tgamma+\theta}(\OmegaT)$, $\oP \in L^\delta(\OmegaT)$ with
    \begin{equation*}
        \partial_t c \in L^2(\OmegaT), \quad 
        \rho \uvec \in C_{\mathrm{w}}\bigl([0,T];L^{\frac{2\tgamma}{\tgamma+1}}(\Omega;\RR^3)\bigr),
    \end{equation*}
    and 
    \begin{equation*}
        \rho(t,x) \geq 0,\quad \mathrm{spt}\, \nu_{(t,x)} \subseteq [0,\infty)
    \end{equation*} 
    for almost all $(t,x) \in \Omega_T$, such that, after passing to a non-relabeled subsequence,
    \begin{equation*}
    \begin{aligned}
        &\nu_{\rho_n} \weakstar \nu \quad \text{in } L^\infty_{\mathrm{w}^*}(\Omega_T; \mathcal{M}(\RR)), 
        \quad 
        \rho_n \to \rho \quad \text{in } C_{\mathrm{w}}([0,T];L^\tgamma(\Omega)), \\
        &\uvec_n \weak \uvec \quad \text{in } L^2(0,T;H^{1}_0(\Omega;\RR^3)),
        \quad 
        c_n \to c \quad \text{in } L^2(0,T;H^{1}(\Omega)),\\ 
        &\rho_n\uvec_n \to \rho \uvec \quad \text{in } C_{\mathrm{w}}([0,T];L^{\frac{2\tgamma}{\tgamma+1}}(\Omega;\RR^3)),
        \quad 
        \Peff(\rho_n) \weak \oP \quad \text{in } L^{\delta}(\OmegaT),\\
        &\rho_n\nabla c_n \weak \rho\nabla c \quad \text{in } L^2(0,T;L^\frac{6\tgamma}{\tgamma+6}(\Omega)),
        \quad 
        \rho^0_n \weak \rho^0 \quad \text{in } L^\tgamma(\Omega).
    \end{aligned}      
    \end{equation*}
    Moreover, $(\rho,\uvec,c)$ satisfies $\eqref{thm: Main Result Continuity Eq}$--$\eqref{thm: Main Result Parabolic Eq}$ and we have for almost all $(t,x) \in \Omega_T$ that
    \begin{equation}\label{thm: Proof Main Result Part I : Relations}
        \rho(t,x) = \int_{[0,\infty)} \xi \, \dd \nu_{(t,x)} (\xi), 
        \quad 
        \oP(t,x) = \int_{[0,\infty)} \Peff(\xi) \, \dd \nu_{(t,x)}(\xi)
    \end{equation}
    and for almost all $x \in \Omega $ that
    \begin{equation}\label{thm: Proof Main Result Part I : Relations Initially}
        \mathrm{spt}\, \nu^0_x \subseteq [0,\infty), \quad \rho^0(x) = \int_{[0,\infty)} \xi \, \dd \nu^0(x).
    \end{equation}
\end{proposition}
\begin{proof}
    By Proposition~\ref{Prop: Uniform-in-n Bounds}, we have that ${(\rho_n)}_{n\in\NN}$ is uniformly bounded in $L^\infty(0,T;L^\tgamma(\Omega))\cap L^{\tgamma+\theta}(\OmegaT)$ and therefore, by Proposition~\ref{Prop : Young Measures}, there exists some Young measure $\nu \in \mathcal{Y}(\Omega_T;\RR)$, such that, after passing to a non-relabeled subsequence,
    \begin{equation*}
        \nu_{\rho_n}\weakstar \nu \quad \text{in } L^\infty(\Omega_T;\mathcal{M}(\RR)).
    \end{equation*}
    By Proposition~\ref{Prop: Uniform-in-n Bounds}, we have that the sequences ${(\rho_n)}_{n\in\NN}$ and ${(\rho_n\uvec_n)}_{n\in\NN}$ are uniformly bounded in $L^\infty(0,T;L^\tgamma(\Omega))\cap L^{\tgamma+\theta}(\OmegaT)$ and $L^\infty(0,T;L^\frac{2\tgamma}{\tgamma+1}(\Omega;\RR^3))$, respectively, and uniformly continuous in $W^{-1,\frac{2\tgamma}{\tgamma+1}}(\Omega)$ and $W^{-1,q}(\Omega;\RR^3)$, respectively, where $q >1$ is defined as in Proposition~\ref{Prop: Uniform-in-n Bounds}.
    Thus, by Lemma~\ref{Lemma : Weak Continuity in Time}, there exist two functions $\rho \in C_{\mathrm{w}}([0,T];L^\tgamma(\Omega))$ and $\overline{\rho\uvec}\in C_{\mathrm{w}}([0,T];L^{\frac{2\tgamma}{\tgamma+1}}(\Omega;\RR^3))$, such that, after passing to a non-relabeled subsequence,
    \begin{equation}\label{thm: Proof Main Result Part I : C_w Conv for rho and rho u}
        \rho_n \to \rho \quad \text{in } C_{\mathrm{w}}([0,T];L^\tgamma(\Omega)),
        \quad
        \rho_n \uvec_n \to \overline{\rho\uvec} \quad \text{in } C_{\mathrm{w}}([0,T];L^\frac{2\tgamma}{\tgamma+1}(\Omega;\RR^3)).
    \end{equation}
    By the Sobolev embedding theorem, we have that the embeddings
    \begin{equation*}
        L^\tgamma(\Omega) \hookrightarrow \hookrightarrow H^{-1}(\Omega), \quad 
        L^\frac{2\tgamma}{\tgamma+1}(\Omega) \hookrightarrow\hookrightarrow H^{-1}(\Omega)
    \end{equation*}
    are compact.
    With Lemma~\ref{Lemma : Weak Continuity in Time} we conclude
    \begin{equation}\label{thm: Proof Main Result Part I : Strong Conv for rho and rho u}
        \rho_n \to \rho \quad \text{in } L^2(0,T;H^{-1}(\Omega)), \quad 
        \rho_n\uvec_n \to \overline{\rho\uvec} \quad \text{in } L^2(0,T;H^{-1}(\Omega;\RR^3)).
    \end{equation}
    Furthermore, by Proposition~\ref{Prop: Uniform-in-n Bounds} and the Banach--Alaoglu theorem, we conclude that there exist functions $\uvec \in L^2(0,T;H^{1}_0(\Omega;\RR^3))$, $c \in L^2(0,T;H^{2}(\Omega))$ with $\partial_t c \in L^2(\OmegaT)$, $\oP \in L^\delta(\OmegaT)$, $\overline{\rho\uvec\otimes\uvec} \in L^2(0,T;L^\frac{3\tgamma}{\tgamma+3}(\Omega;\RR^{3\times 3}))$, $\overline{\rho\nabla c} \in L^2(0,T;L^{\frac{6\tgamma}{\tgamma+6}}(\Omega;\RR^3))$, such that, after passing to a non-relabeled subsequence,
    \begin{equation}\label{thm: Proof Main Result Part I : Weak Conv for u}
        \uvec_n \weak \uvec \quad \text{in } L^2(0,T;H^{1}_0(\Omega;\RR^3)),
    \end{equation}
    \begin{equation}\label{thm: Proof Main Result Part I : Weak Conv for rho u x u}
        \rho_n \uvec_n \otimes \uvec_n \weak \overline{\rho\uvec\otimes\uvec} \quad \text{in } L^2(0,T;L^\frac{3\tgamma}{\tgamma+3}(\Omega;\RR^{3\times 3})),
    \end{equation}
    \begin{equation}\label{thm: Proof Main Result Part I : Weak Conv for P}
        \Peff(\rho_n) \weak \oP \quad \text{in } L^\delta(\OmegaT),
    \end{equation}
    \begin{equation}\label{Thm: Proof Main Result Part I : weak convergence of rho_n nabla c_n}
        \rho_n\nabla c_n \weak \overline{\rho\nabla c} \quad \text{in } L^2(0,T;L^{\frac{6\tgamma}{\tgamma+6}}(\Omega;\RR^3))
    \end{equation}
    and
    \begin{equation}\label{thm: Proof Main Result Part I : Weak Conv for c and partial_t c}
        c_n \weak c \quad \text{in } L^2(0,T;H^{2}(\Omega)), \quad 
        \partial_t c_n \weak \partial_tc \quad \text{in } L^2(\OmegaT).
    \end{equation}
    By standard results on Bochner spaces (see e.g.~\cite{EPDE}), we conclude from the regularity of $c$ that $c \in C([0,T];H^{1}(\Omega))$.
    Moreover, as we have by the Sobolev embedding theorem $H^{2}(\Omega) \hookrightarrow \hookrightarrow H^{1}(\Omega) \hookrightarrow L^2(\Omega)$ where the first embedding is compact, we conclude from the uniform bounds in Proposition~\ref{Prop: Uniform-in-n Bounds} by applying the Aubin--Lions lemma (Lemma~\ref{thm: Aubin Lions}) that
    \begin{equation}\label{thm: Proof Main Result Part I : Strong Conv for c}
        c_n \to c \quad \text{in } L^2(0,T;H^{1}(\Omega)).
    \end{equation}
    By Lemma~\ref{lem: Product Convergences} we conclude with the convergences $\eqref{thm: Proof Main Result Part I : Strong Conv for rho and rho u}$ and $\eqref{thm: Proof Main Result Part I : Weak Conv for u}$ that 
    \begin{equation}\label{Thm: Proof Main Result Part I : m = rho u}
        \overline{\rho\uvec} = \rho \uvec \quad \text{a.e. in } \Omega_T.
    \end{equation}
    From this relation and the convergences $\eqref{thm: Proof Main Result Part I : Strong Conv for rho and rho u}$ and $\eqref{thm: Proof Main Result Part I : Weak Conv for u}$, we conclude again by Lemma~\ref{lem: Product Convergences} that
    \begin{equation}\label{Thm: Proof Main Result Part I : MM = rho u x u}
        \overline{\rho\uvec\otimes\uvec} = \rho \uvec \otimes \uvec \quad \text{a.e. in } \Omega_T.
    \end{equation}
    Combining the strong convergence in $\eqref{thm: Proof Main Result Part I : Strong Conv for c}$ with the weak convergence in $\eqref{thm: Proof Main Result Part I : C_w Conv for rho and rho u}$, we obtain
    \begin{equation}\label{Thm: overline rho nabla c}
        \overline{\rho\nabla c} = \rho \nabla c \quad \text{a.e. in } \OmegaT.
    \end{equation}
    As $\rho_n\geq 0$ almost everywhere on $\OmegaT$, we conclude by the weak convergence of $(\nu_{\rho_n})_{n\in\NN}$ for any $b \in C^0_0((-\infty,0))$ and any $\varphi \in L^1(\OmegaT)$ that
    \begin{align*}
        0 =\int_{\Omega_T}\varphi\, b(\rho_n)\, \dd x \, \dd t \to
        \int_{\Omega_T}\varphi \, \langle \nu,b\rangle \, \dd x \, \dd t.
    \end{align*}
    This implies for almost all $(t,x) \in \OmegaT$ that $\langle \nu_{(t,x)},b\rangle = 0$ for any test function $b \in C^0_0((-\infty,0))$ and thus $\mathrm{spt}\, \nu_{(t,x)}\subseteq [0,\infty)$.
    The relations in $\eqref{thm: Proof Main Result Part I : Relations}$ follow from Proposition~\ref{Prop : Young Measures} since ${(\rho_n)}_{n\in\NN}$ is uniformly bounded in $L^{\tgamma+\theta}(\OmegaT)$ and ${(\Peff(\rho_n))}_{n\in\NN}$ is uniformly bounded in $L^\delta(\OmegaT)$ (cf.\ Proposition~\ref{Prop: Uniform-in-n Bounds}).
    In particular, we have $\rho\geq 0$.
    By the same arguments, we conclude for the initial density sequence $(\rho^0_n)_{n\in\NN}$ that
    \begin{align*}
        \rho^0_n \weak \rho^0 \quad \text{in } L^\tgamma(\Omega)
    \end{align*}
    with $\nu^0$ and $\rho^0$ satisfying the relations in $\eqref{thm: Proof Main Result Part I : Relations Initially}$.
    For any $n \in \NN$, we have that the integral identities $\eqref{Def: Renormalized Finite Energy Weak Solutions Continuity Eq}$--$\eqref{Def: Renormalized Finite Energy Weak Solutions Parabolic Eq}$ hold for $(\rho_n,\uvec_n,c_n)$.
    The convergences $\eqref{thm: Proof Main Result Part I : C_w Conv for rho and rho u}$--$\eqref{thm: Proof Main Result Part I : Weak Conv for c and partial_t c}$
    together with the convergences that hold for the initial data by hypothesis are strong enough to pass to the limit $n \to \infty$ in these identities.
    With $\eqref{Thm: Proof Main Result Part I : m = rho u}$--$\eqref{Thm: overline rho nabla c}$, we obtain in the limit $n \to \infty$, that $(\rho,\uvec,c)$ satisfies precisely $\eqref{thm: Main Result Continuity Eq}$--$\eqref{thm: Main Result Parabolic Eq}$.
\end{proof}
By Proposition~\ref{thm: Proof Main Result Part I}, we have extracted some effective deterministic quantities $(\rho,\uvec,c)$ that satisfy the effective system $\eqref{thm: Main Result Continuity Eq}$--$\eqref{thm: Main Result Parabolic Eq}$.
However, this system is unclosed in the sense that the effective pressure $\oP$ has lost its relation to $(\rho,\uvec,c)$.
This is due to the non-linear pressure function $\Peff$ and the strong oscillations that we expect for the sequence of densities ${(\rho_n)}_{n\in\NN}$.
With the relation in $\eqref{thm: Proof Main Result Part I : Relations}$, we can close the quantity $\oP$ in terms of the effective probabilistic quantity $\nu$, for which we have not extracted some effective equation yet.
Thus, we close the system, by finding an effective equation for $\nu$.
Since $\nu$ is the weak-$*$ limit of the sequence ${(\nu_{\rho_n})}_{n\in\NN}$ and $\nu_{\rho_n}$ is defined via
\begin{equation}\label{Definition of nu rho_n}
    \langle (\nu_{\rho_n})_{(t,x)}, b\rangle = b(\rho_n(t,x)) \quad \text{for a.e. } (t,x) \in \Omega_T,
\end{equation}
for any $b \in C^1_0(\RR)$, we find such an effective equation by computing the time derivative of $\eqref{Definition of nu rho_n}$ and passing to the limit $n \to \infty$ in the resulting  equation.
The time derivative is exactly the renormalized continuity equation 
\begin{equation*}
    \partial_t b(\rho_n) + \div\bigl( b(\rho_n) \uvec_n \bigr) + \bigl(b^\prime(\rho_n)\rho_n-b(\rho_n)\bigr) \div\uvec_n = 0.
\end{equation*}
In order to pass to the limit $n\to \infty$ in this equation, we need the following compactness result that is a consequence of a weak compactness property for the so called effective viscous flux
\begin{equation*}
    \Peff(\rho) - (\lambda + 2\mu) \div\uvec,
\end{equation*}
which is known from the theory on the compressible Navier--Stokes equations with constant viscosity coefficients (see e.g.\ \cite{FN,FNP,NASI}).
More specifically, interpreting the additional term $\rho \nabla c$ in the momentum equation as a forcing term
and applying a general version of the weak compactness property for the effective viscous flux proven in \cite{NASI} (cf.\ Theorem~\ref{thm: Effective Viscous Flux}) yields the following result.
\begin{lemma}\label{Lemma : Preliminary for Main Proof}
    Under the hypotheses and notations of Proposition~\ref{thm: Proof Main Result Part I}, we have for any $b \in C^1_c(\RR)$ and any $n \in \NN$ that $b(\rho_n) \in L^\infty(0,T;L^\infty(\Omega))\cap C_{\mathrm{w}}([0,T];L^q(\Omega))$ for any $q \in [1,\infty)$.
    Moreover, we have that
    \begin{equation}\label{Lemma : Preliminary for Main Proof : Weak Convergence}
        b(\rho_n) \weakstar \overline b \quad \text{in }L^\infty(0,T;L^\infty(\Omega)),\quad b(\rho_n) \to \overline{b} \quad \text{in } C_{\mathrm{w}}([0,T];L^q(\Omega)) \quad \forall\, q \in [1,\infty)
    \end{equation}
    with 
    \begin{equation}\label{Lemma : Preliminary for Main Proof : Representation Weak Limit}
        \overline{b} (t,x) = \int_\RR b(\xi) \, \dd \nu_{(t,x)}(\xi) \quad \text{for a.e. } (t,x) \in \OmegaT,
    \end{equation}
    and in particular, we have
    \begin{equation}\label{Lemma : Preliminary for Main Proof : Strong Convergence}
        b(\rho_n) \to \overline{b} \quad \text{in } L^2(0,T;H^{-1}(\Omega)).
    \end{equation}
    Moreover, the effective viscous flux identity
    \begin{equation}\label{Lemma : Preliminary for Main Proof : Effective Viscous Flux}
    \begin{aligned}
        &\lim\limits_{n\to\infty} \int_0^T \int_\Omega 
        \varphi \biggl( \Peff(\rho_n) - (\lambda+2\mu) \div\uvec_n \biggr) b(\rho_n) \, \dd x \, \dd t\\
        &\qquad= 
        \int_0^T \int_\Omega \varphi \biggl( \oP -(\lambda+2\mu) \div \uvec \biggr) \overline{b}\, \dd x \, \dd t
    \end{aligned}
    \end{equation}
    holds for any test function $\varphi \in \mathcal{D}( \Omega_T )$.\\
    In particular, we have
    \begin{equation}\label{Lemma : Preliminary for Main Proof : Weak Convergence b(rho_n)div(u_n)}
        b(\rho_n) \div\uvec_n 
        \weak 
        \overline{b}\div\uvec 
        +
        \frac{1}{\lambda+2\mu}\Bigl( \overline{\Peff b} - \oP\, \overline{b}\Bigr)
        \quad \text{in } \mathcal{D}^\prime(\OmegaT),
    \end{equation}
    with 
    \begin{equation}\label{pw expr Peff b}
        \overline{\Peff b}(t,x) = \int_\RR \Peff(\xi)b(\xi) \, \dd  \nu_{(t,x)}(\xi) \quad \text{for a.e. } (t,x) \in \OmegaT,
    \end{equation}
    where $\oP\in L^\delta(\OmegaT)$ denotes the limit pressure in Proposition~\ref{thm: Proof Main Result Part I}.
\end{lemma}
\begin{proof}
    As $\nu_{\rho_n}\weakstar \nu$ in $L^\infty(\Omega_T;\mathcal{M}(\RR))$ and $b \in C^1_c(\RR) \subseteq C^0_0(\RR)$, we have that
    \begin{equation*}
        b(\rho_n) \weakstar \overline{b} \quad \text{in } L^\infty(0,T;L^\infty(\Omega))
    \end{equation*}
    with $\overline{b}$ satisfying the relation $\eqref{Lemma : Preliminary for Main Proof : Representation Weak Limit}$ (cf.\ Proposition~\ref{Prop : Young Measures}).
    For any $n \in \NN$, we have that $(\rho_n,\uvec_n,c_n)$ satisfies the renormalized continuity equation 
    \begin{equation}\label{Lemma : Preliminary for Main Proof : Renormalized Continuity Equation}
        \partial_t b(\rho_n) + \div\bigl( b(\rho_n) \uvec_n) = -B(\rho_n) \div\uvec_n =: \mathcal{G}_n \quad \text{in } \mathcal{D}^\prime\bigl( \Omega_T \bigr),
    \end{equation}
    where $B\in C^1_c(\RR)$ is defined via $B(r) := b^\prime(r)r-b(r)$ for any $r \in \RR$.
    In particular, we have that 
    \begin{equation*}
        \frac{\dd}{\dd t} \int_\Omega b(\rho_n) \phi \, \dd x 
        = 
        \int_\Omega b(\rho_n) \uvec_n\cdot \nabla \phi \, \dd x 
        -
        \int_\Omega B(\rho_n) \div\uvec_n \phi \, \dd x
        \quad \text{in } \mathcal{D}^\prime\bigl( (0,T) \bigr)
    \end{equation*}
    for any test function $\phi \in \mathcal{D}(\Omega)$.
    For any fixed $\phi \in \mathcal{D}(\Omega)$, the right hand side of this equation lies in $L^2\bigl( (0,T) \bigr)$.
    Thus, we conclude with Lemma~\ref{Lemma : Criterion for Weak Continuity in Time}, that (after redefining $b(\rho_n)$ on a set of Lebesgue-measure zero) 
    \begin{equation*}
        b(\rho_n) \in C_{\mathrm{w}}([0,T];L^q(\Omega)) \quad \forall \, q \in [1,\infty),
    \end{equation*}
    and, that
    \begin{equation*}
        \int_\Omega \bigl( b(\rho_n)(t) - b(\rho_n)(s)\bigr) \phi\,\dd x
        =
        \int_s^t \int_\Omega 
        \bigl(
        b(\rho_n) \uvec_n \cdot \nabla \phi -B(\rho_n) \div\uvec_n  \phi 
        \bigr)
        \, \dd x \, \dd \tau
    \end{equation*}
    for any $0\leq s \leq t \leq T$.
    By using the uniform estimates in Proposition~\ref{Prop: Uniform-in-n Bounds}, we estimate the right-hand side of this equation via Hölder's inequality as
    \begin{equation*}
        \begin{aligned}
            &\biggl|\int_s^t \int_\Omega 
            \bigl(
            B(\rho_n) \div\uvec_n \, \phi - b(\rho_n) \uvec_n \cdot \nabla \phi
            \bigr)
            \, \dd x \, \dd \tau \biggr|
            \leq \curlyC_0 \sqrt{t-s} \|\phi\|_{H_0^1(\Omega)}.
        \end{aligned}
    \end{equation*}
    In particular, we have that $b(\rho_n)$ is uniformly continuous in $H^{-1}(\Omega)$. 
    As we have already verified that ${(b(\rho_n))}_{n\in\NN}$ is uniformly bounded in $L^\infty(0,T;L^\infty(\Omega))$, we conclude with Lemma~\ref{Lemma : Weak Continuity in Time}, that (after possibly redefining $\overline{b}$ on a set of Lebesgue-measure zero)
    \begin{equation*}
        b(\rho_n) \to \overline{b} \quad \text{in } C_{\mathrm{w}}([0,T];L^q(\Omega)) \quad \forall \, q \in [1,\infty)
    \end{equation*}
    and in particular, as $L^q(\Omega) \hookrightarrow\hookrightarrow H^{-1}(\Omega)$ is compact for $q >\frac{6}{5}$,
    \begin{equation}\label{strong cv bn}
        b(\rho_n) \to \overline{b} \quad \text{in } L^2(0,T;H^{-1}(\Omega)).
    \end{equation}
    From $\eqref{Lemma : Preliminary for Main Proof : Renormalized Continuity Equation}$, $\eqref{strong cv bn}$ and the fact that $\uvec_n \weak \uvec$ in $L^2(0,T;H^1_0(\Omega;\RR^3))$ (cf.\ Proposition~\ref{thm: Proof Main Result Part I}), we conclude that
    \begin{align}\label{cv fn to F distr}
        \mathcal{G}_n \weak \partial_t \overline{b} + \div(\overline{b}\uvec)=:\overline{\mathcal{G}} \quad \text{in } \mathcal{D}^\prime(\OmegaT).
    \end{align}
    By the uniform bounds from Proposition~\ref{Prop: Uniform-in-n Bounds}, we have that
    \begin{align*}
        |\langle \overline{\mathcal{G}}, \varphi \rangle_{\mathcal{D}^\prime,\, \mathcal{D}}|
        \leq 
        \limsup\limits_{n\to\infty} |\langle \mathcal{G}_n, \varphi \rangle _{\mathcal{D}^\prime,\,\mathcal{D}}|
        \leq 
        \curlyC_0 \|\varphi \|_{L^2(\OmegaT)}
    \end{align*}
    for any test function $\varphi \in \mathcal{D}(\OmegaT)$ and by duality we conclude $\overline{\mathcal{G}} \in L^2(\OmegaT)$.
    Combining the fact that $(\mathcal{G}_n)_{n\in\NN}$ is uniformly bounded in $L^2(\OmegaT)$ with the density of $\mathcal{D}(\OmegaT)$ in $L^2(\OmegaT)$, we conclude from $\eqref{cv fn to F distr}$ that
    \begin{align*}
        \mathcal{G}_n \weak \overline{\mathcal{G}} \quad \text{in } L^2(\OmegaT).
    \end{align*}
    With the convergences already proved in Proposition~\ref{thm: Proof Main Result Part I}, we can apply the general form of the effective viscous flux Lemma (cf.\ Theorem~\ref{thm: Effective Viscous Flux}) with
    \begin{equation*}
    \begin{aligned}
        &\mathbf{q}_n= (\rho_n\uvec_n), \quad \mathbf{q}=(\rho\uvec), \quad z =\frac{2\tgamma}{\tgamma+1}\in (\frac{6}{5},\infty),\\
        &p_n = \Peff(\rho_n), \quad p=\oP, \quad r=\delta \in (1,\infty),\\
        &\mathbf{F_n}=\coup\rho_n \nabla c_n, \quad \mathbf{F}= \coup \rho \nabla c, \quad s=\frac{3}{2}\in (1,\infty),\\
        &f_n=\mathcal{G}_n, \quad f=\overline{\mathcal{G}}, \quad
        g_n = b(\rho_n), \quad g=\overline{b}
    \end{aligned}
    \end{equation*}
    in order to deduce the effective viscous flux identity $\eqref{Lemma : Preliminary for Main Proof : Effective Viscous Flux}$.\\
    Let us now verify the relation $\eqref{Lemma : Preliminary for Main Proof : Weak Convergence b(rho_n)div(u_n)}$.    
    We have $\Peff b\in C^0_c(\RR)$, where $\Peff$ is extended by zero onto $(-\infty,0)$. 
    Thus, the convergence $\nu_{\rho_n} \weakstar \nu$ in $L^\infty_{\mathrm{w}^*}(\Omega_T;\mathcal{M}(\RR))$ yields
    \begin{equation*}
        \Peff(\rho_n) b(\rho_n) \weakstar \overline{\Peff b} 
        \quad \text{in } L^\infty(0,T;L^\infty(\Omega))
    \end{equation*}
    with $\overline{\Peff b}$ satisfying the relation $\eqref{pw expr Peff b}$.
    From $\eqref{strong cv bn}$ and the weak convergence $\uvec_n\weak \uvec$ in $L^2(0,T;H^1_0(\Omega;\RR^3))$ we conclude
    \begin{align*}
        b(\rho_n)\uvec_n \weak \overline{b}\uvec \quad \text{in } \mathcal{D}^\prime(\OmegaT).
    \end{align*}
    Combining these convergences with the effective viscous flux identity $\eqref{Lemma : Preliminary for Main Proof : Effective Viscous Flux}$ yields $\eqref{Lemma : Preliminary for Main Proof : Weak Convergence b(rho_n)div(u_n)}$.
\end{proof}

With Lemma~\ref{Lemma : Preliminary for Main Proof} at hand, we complete the proof of our main result Theorem~\ref{thm: Main Result}.

\begin{proof}[Proof of Theorem~\ref{thm: Main Result}]
    Let the hypotheses and notations of Proposition~\ref{thm: Proof Main Result Part I} hold true.
    We only have to show that the limit Young measure $\nu$ satisfies $\eqref{thm: Main Result Kinetic Eq}$.
    Let us fix some arbitrary test function $\psi \in C^1_c([0,T)\times\Omega\times\RR)$.
    By a density argument, we can assume without loss of generality that 
    \begin{equation*}
        \psi(t,x,\xi) = \varphi(t,x) b(\xi) \qquad \forall\, (t,x,\xi) \in \Omega_T\times\RR,
    \end{equation*}
    for some $\varphi \in C^1_c([0,T)\times\Omega)$ and $b \in C^1_c(\RR)$.
    For any $n \in \NN$, we have that $(\rho_n,\uvec_n,c_n)$ satisfies the renormalized continuity equation $\eqref{Def: Renormalized Finite Energy Weak Solutions Renormalized Continuity Eq}$. 
    Thus, we have 
    \begin{align}\label{Proof Main Result Part II : Master Eq}
        \int_0^T \int_\Omega 
        \bigl(
        b(\rho_n) \partial_t \varphi 
        +
        b(\rho_n)\uvec_n \cdot \nabla \varphi
        -
        B(\rho_n) \div\uvec_n \, \varphi
        \bigr)
        \, \dd x \, \dd t
        =
        -\int_\Omega b(\rho^0_n)\varphi(0,x) \, \dd x
    \end{align}
    for any $n \in \NN$, where $B(r):=b^\prime(r)r-b(r)$ for any $r \in \RR$.\\
    By hypothesis, we have that $\nu_{\rho^0_n}\weakstar \nu^0$ in $L^\infty_{\mathrm{w}^*}(\Omega;\mathcal{M}(\RR))$, and thus
    \begin{equation}\label{Proof Main Result Part II : Conv IC}
        -\int_\Omega \varphi(0,\cdot) b(\rho^0_n)  \, \dd x \to -\int_\Omega \varphi(0,\cdot)  \langle\nu^0,b\rangle \, \dd x. 
    \end{equation}
    From the convergences $\eqref{Lemma : Preliminary for Main Proof : Weak Convergence}, \eqref{Lemma : Preliminary for Main Proof : Weak Convergence b(rho_n)div(u_n)}$ and the relation $\eqref{Lemma : Preliminary for Main Proof : Representation Weak Limit}$ in Lemma~\ref{Lemma : Preliminary for Main Proof}, we conclude that
    \begin{equation}\label{Proof Main Result Part II : Conv partial_t b(rho)}
        \int_0^T \int_\Omega b(\rho_n) \partial_t \varphi \, \dd x \, \dd t 
        \to 
        \int_0^T \int_\Omega \langle \nu,b\rangle \partial_t \varphi \, \dd x \, \dd t
    \end{equation}
    and
    \begin{equation}\label{Proof Main Result Part II : Conv B div u}
    \begin{aligned}
        -\int_0^T \int_\Omega B(\rho_n) \div\uvec_n\, \varphi\, \dd x \, \dd t
        \to
        &-\int_0^T \int_\Omega \langle \nu,B \rangle \div\uvec\,\dd x \, \dd t\\
        &\quad + \frac{1}{\lambda+2\mu}\int_0^T \int_\Omega \Bigr( \langle \nu, B\rangle \oP - \langle \nu , B \Peff\rangle\Bigl)\varphi \, \dd x \, \dd t.
    \end{aligned}
    \end{equation}
    By definition of $\osc$ (cf.\ $\eqref{Defi Osc}$), we can rewrite the second integral on the right hand side in $\eqref{Proof Main Result Part II : Conv B div u}$ as
    \begin{equation}\label{Proof Main Result Part II : Rewriting in Q}
    \frac{1}{\lambda+2\mu}
        \int_0^T \int_\Omega \Bigr(\langle \nu, B\rangle \oP - \langle \nu , B \Peff\rangle\Bigl)\varphi \, \dd x \, \dd t
        =
        \int_0^T \int_\Omega \langle \nu , \bigl( \xi \partial_\xi b - b \bigr) \osc \rangle \varphi \, \dd x \, \dd t.
    \end{equation}
    Using the convergence $\eqref{Lemma : Preliminary for Main Proof : Strong Convergence}$ and the fact that $\uvec_n \weak \uvec$ in $L^2(0,T;H^{1}_0(\Omega;\RR^3))$, we conclude with Lemma~\ref{lem: Product Convergences} 
    \begin{equation}\label{Proof Main Result Part II : Conv b(rho) u}
        \int_0^T \int_\Omega b(\rho_n) \uvec_n \cdot \nabla \varphi\, \dd x \, \dd t
        \to
        \int_0^T \int_\Omega \langle \nu,b\rangle \uvec \cdot \nabla\varphi \, \dd x \, \dd t.
    \end{equation}
    With the convergences $\eqref{Proof Main Result Part II : Conv IC}, \eqref{Proof Main Result Part II : Conv partial_t b(rho)}, \eqref{Proof Main Result Part II : Conv B div u}, \eqref{Proof Main Result Part II : Conv b(rho) u}$ and the relation $\eqref{Proof Main Result Part II : Rewriting in Q}$, we obtain that equation $\eqref{Proof Main Result Part II : Master Eq}$ yields in the limit $n \to \infty$ 
    \begin{align*}
        \int_0^T \int_\Omega 
        \Bigl\langle 
        \nu,
        \partial_t \psi
        +
        \nabla \psi \cdot \uvec 
        -
        \bigl(\xi \partial_\xi \psi - \psi\bigr)\div\uvec
        +
        \bigl( \xi \partial_\xi \psi - \psi\bigr) \osc 
        \Bigr\rangle 
        =
        -\int_\Omega \Bigl\langle \nu^0,\psi(0,\cdot,\cdot)\Bigr\rangle \, \dd x.
    \end{align*}
    This is exactly $\eqref{thm: Main Result Kinetic Eq}$. The proof is now complete.
\end{proof}

We conclude this section with the proof of Corollary~\ref{Cor : Main Result}.
The proof relies on the following technical lemma, that follows from an integration by parts rule for the Lebesgue--Stieltjes integral for functions of bounded variations.
For a more detailed exposition on this integration by parts rule, we refer to \cite[Chapter~1]{PS}.

\begin{lemma}\label{Lemma : Auxiliary for L-S Integral}
    Let $\mu \in \mathcal{M}^+(\RR)$ be a non-negative finite Radon-measure. 
    Then the function
    \begin{equation*}
        f\colon \RR \to \RR, \qquad \xi \mapsto f(\xi):=\mu((-\infty,\xi])
    \end{equation*}
    is non-decreasing, right-continuous and satisfies
    \begin{equation}\label{Lemma : Auxiliary for L-S Integral : Properties of f} 
        \lim\limits_{s \nearrow \infty} f(s) = \|\mu\|_{\mathcal{M}(\RR)},
        \qquad
        \lim\limits_{s \searrow -\infty } f(s) = 0.
    \end{equation}
    Furthermore, we have
    \begin{equation}\label{Lemma : Auxiliary for L-S Integral : Integration by Parts I}
        \int_\RR b(\xi)\, \dd \mu (\xi) 
        =
        -\int_\RR b^\prime(\xi) f(\xi) \, \dd \xi
    \end{equation}
    for any test function $b \in C^1_c(\RR)$.
    Moreover, we have for any non-negative function $h \in C^0(\RR)$ that satisfies
    \begin{equation*}
        \int_\RR h(\xi) \, \dd \mu(\xi) < \infty
    \end{equation*}
    that
    \begin{equation}\label{Lemma : Auxiliary for L-S Integral : Integration by Parts II}
        \int_\RR b(\xi) h(\xi) \, \dd \mu (\xi)
        =
        -\int_\RR b^\prime(\xi) \biggl( \int_{(-\infty,\xi]}h(\eta)\, \dd \mu(\eta) \biggr) \, \dd \xi
    \end{equation}
    for any test function $b \in C^1_c(\RR)$.
\end{lemma}
\begin{proof}
The relations in $\eqref{Lemma : Auxiliary for L-S Integral : Properties of f}$ follow from elementary properties of non-negative measures (see e.g.~\cite[Chapter 1]{R}) and for the proof of relation $\eqref{Lemma : Auxiliary for L-S Integral : Integration by Parts I}$ we refer to \cite[Lemma 1.3.7]{PS}.
The relation $\eqref{Lemma : Auxiliary for L-S Integral : Integration by Parts II}$ follows from $\eqref{Lemma : Auxiliary for L-S Integral : Integration by Parts I}$ by considering the non-negative finite Radon-measure $\mu_h \in \mathcal{M}^+(\RR)$ that is defined via
\begin{align*}
    \langle \mu_h,b\rangle := \int_\RR b(\xi) h(\xi) \, \dd \mu(\xi) \quad \forall\, b \in  C^0_0(\RR).
\end{align*}
\end{proof}
With the relations from Lemma~\ref{Lemma : Auxiliary for L-S Integral} we conclude Corollary~\ref{Cor : Main Result} from Theorem~\ref{thm: Main Result}.
The proof relies on an appropriate choice of test functions and follows the technique demonstrated in \cite{PS}.
\begin{proof}[Proof of Corollary~\ref{Cor : Main Result}]
    For almost all $(t,x)\in \Omega_T$, we have that
    \begin{equation}\label{Cor : Main Result : spt nu}
        \nu_{(t,x)}\in \mathcal{P}(\RR), \quad \mathrm{spt}\,\nu_{(t,x)}\subseteq [0,\infty),
    \end{equation}
    and thus, by Lemma~\ref{Lemma : Auxiliary for L-S Integral}, $\xi \mapsto f(t,x,\xi)$ is non-decreasing, right-continuous and satisfies
    \begin{equation}\label{Cor : Main Result : CDF f}
        f(t,x,\xi) = 0\quad \forall\, \xi \in (-\infty,0), \qquad 
                \lim\limits_{\xi \nearrow \infty} f(t,x,\xi)=1.
    \end{equation}
    By definition of $f$, we have for almost all $(t,x) \in \Omega_T$ that the Lebesgue--Stieltjes measure corresponding to $f(t,x,\cdot)$ is exactly given by $\nu_{(t,x)}$ and therefore we have
    \begin{equation}\label{Cor : Main Result : Integration Relation}
        \int_\RR b(\xi)\, \dd f(t,x,\xi) = \int_\RR b(\xi) \, \dd \nu_{(t,x)}(\xi)
    \end{equation}
    for any function $b$ that is integrable with respect to $\nu_{(t,x)}$.
    Also, we have that $\oP \in L^\delta(\Omega_T)$ and thus
    \begin{equation*}
       \int_\RR \Peff(\xi) \, \dd \nu_{(t,x)}(\xi) < \infty
    \end{equation*}
    for almost all $(t,x) \in \Omega_T$.
    Thus, by Lemma~\ref{Lemma : Auxiliary for L-S Integral} and $\eqref{Cor : Main Result : Integration Relation}$, we have for almost all $(t,x) \in \Omega_T$ the integration by parts rules
    \begin{align}
        &\int_\RR b(\xi) \, \dd \nu_{(t,x)}(\xi) = - \int_\RR \partial_\xi b(\xi) f(t,x,\eta)\, \dd \xi,\label{Cor : Main Result : Integration by Parts I} \\
        &\int_\RR b(\xi) \Peff(\xi) \, \dd \nu_{(t,x)}(\xi)
        =
        -\int_\RR \partial_\xi b(\xi) \biggl( \int_{(-\infty,\xi]} \Peff(\eta) \, \dd f(t,x,\eta) \biggr)\, \dd \xi,\label{Cor : Main Result : Integration by Parts II}
    \end{align}
    for any test function $b \in C^1_c(\RR)$.
    Let us fix some test function $\psi \in C^1_c([0,T)\times\Omega\times\RR)$.
    By a density argument, we can assume without loss of generality that $\psi = \varphi  b$ for some $\varphi \in C^1_c([0,T)\times\Omega)$, $b \in C^1_c(\RR)$.
    Let us define a function $B$ via
    \begin{equation*}
        B\colon \RR \to \RR, \qquad 
        \xi \mapsto B(\xi):=\int_{(\xi,\infty)} b(\eta)\, \dd \eta.
    \end{equation*}
    Then we have
    \begin{align}\label{Cor : Main Result : relations for B}
        &\partial_\xi B(\xi) = -b(\xi), \quad 
        \partial_{\xi\xi}B(\xi) = - \partial_\xi b(\xi), \quad 
        \partial_\xi \bigl( \xi \partial_\xi B(\xi)  - B(\xi) \bigr) = - \xi \partial_\xi b(\xi)
    \end{align}
    for any $\xi \in \RR$.
    With the second relation in $\eqref{Cor : Main Result : spt nu}$ and the first relation in $\eqref{Cor : Main Result : CDF f}$, we conclude by a truncation argument that $\eqref{Cor : Main Result : Integration by Parts I}$ and $\eqref{Cor : Main Result : Integration by Parts II}$ hold also for $B$.
    Thus, we have
    \begin{align}
        &\int_\RR B(\xi) \, \dd \nu_{(t,x)}(\xi) 
        =
        \int_\RR b(\xi) f(t,x,\xi) \, \dd \xi,\label{Cor : Main Result : Integration by Parts III}\\
        &\int_\RR \bigl(\xi \partial_\xi B(\xi) - B(\xi) \bigr) \, \dd \nu_{(t,x)}(\xi) = \int_\RR \xi \partial_\xi b(\xi) f(t,x,\xi)\, \dd \xi,\label{Cor : Main Result : Integration by Parts IV}\\
        &\int_\RR \bigl(\xi \partial_\xi B(\xi) - B(\xi) \bigr) \Peff(\xi) \, \dd \nu_{(t,x)}(\xi) = \int_\RR \xi \partial_\xi b(\xi)\biggl( \int_{(-\infty,\xi]}\Peff(\eta)\, \dd f(t,x,\eta) \biggr)\, \dd \xi\label{Cor : Main Result : Integration by Parts V}
    \end{align}
    for almost all $(t,x) \in \Omega_T$, where we have used $\eqref{Cor : Main Result : relations for B}$.
    Corresponding relations hold also for $f^0$.
    With the second relation in $\eqref{Cor : Main Result : spt nu}$ and the first relation in $\eqref{Cor : Main Result : CDF f}$, we have after a truncation argument that the function $\varphi B$ is a valid test function for the kinetic equation $\eqref{thm: Main Result Kinetic Eq}$.
    We obtain
    \begin{equation}\label{Cor : Main Result : Equation to Calculate}
        \int_0^T \int_\Omega 
        \Bigl\langle \nu,
        B\partial_t \varphi 
        + 
        B \nabla \varphi\cdot \uvec 
        -
        \bigl( \xi\partial_\xi B - B\bigr)\varphi\div\uvec
        +
        \bigl(\xi\partial_\xi B - B\bigr) \varphi \osc 
        \Bigr\rangle 
        \, \dd x \, \dd t
        =
        -\int_\Omega 
        \Bigl\langle \nu^0, \varphi(0,\cdot)B
        \Bigr\rangle \, \dd x.
    \end{equation}
    With the relations $\eqref{Cor : Main Result : Integration by Parts III}$ and $\eqref{Cor : Main Result : Integration by Parts IV}$, we infer
    \begin{align*}
        &\int_0^T \int_\Omega \Bigl\langle \nu, B \partial_t\varphi\Bigr\rangle \, \dd x \, \dd t
        =
        \int_0^T \int_\Omega \int_\RR b(\xi)f(t,x,\xi) \partial_t \varphi(t,x) \, \dd \xi\, \dd x \, \dd t,
        \\
        &\int_0^T \int_\Omega \Bigl\langle \nu, B \nabla\varphi \cdot \uvec\Bigr\rangle \, \dd x \, \dd t
        =
        \int_0^T \int_\Omega \int_\RR b(\xi) f(t,x,\xi) \nabla\varphi(t,x) \cdot \uvec (t,x) \, \dd \xi\,\dd x\, \dd t,\\
        &\int_\Omega \Bigl\langle \nu^0, B \varphi(0,\cdot) \Bigr\rangle \, \dd x
        =
        \int_\Omega \int_\RR b(\xi) f^0(x,\xi) \varphi(0,x)\, \dd \xi \, \dd x,
    \end{align*}
    and
    \begin{align*}
        \int_0^T \int_\Omega \Bigl\langle \nu, \bigl(\xi\partial_\xi B -B \bigr) \varphi \div\uvec\Bigr\rangle \, \dd x \, \dd t
        =
        \int_0^T \int_\Omega \int_\RR \xi \partial_\xi b(\xi) f(t,x,\xi) \varphi(t,x)\div\uvec(t,x) \,\dd \xi\, \dd x \, \dd t.
    \end{align*}
    From the relations $\eqref{Cor : Main Result : Integration by Parts IV}$ and $\eqref{Cor : Main Result : Integration by Parts V}$ we infer
    \begin{align*}
        &\int_0^T \int_\Omega 
        \Bigl\langle \nu, \bigl( \xi \partial_\xi B - \xi \bigr) \varphi \osc \Bigr\rangle \, \dd x \, \dd t\\
        &=
        \frac{1}{\lambda + 2 \mu}
        \Biggl(
        \int_0^T \int_\Omega \Bigl\langle \nu, \bigl( \xi \partial_\xi B - \xi \bigr) \oP \varphi \Bigr\rangle \, \dd x \, \dd t  
        -
        \int_0^T \int_\Omega \Bigl\langle \nu, \bigl( \xi \partial_\xi B - \xi \bigr)\Peff \varphi \Bigr\rangle\, \dd x \, \dd t
        \Biggr)\\
        &=
        \frac{1}{\lambda + 2 \mu} 
        \int_0^T \int_\Omega \int_\RR \xi \partial_\xi b(\xi) \varphi(t,x) \Biggl( \oP(t,x) f(t,x,\xi) 
        -
         \int_{(-\infty,\xi]}\Peff(\eta)\, \dd f(t,x,\eta) \Biggr) 
        \, \dd \xi \, \dd x \, \dd t\\
        &=
        \int_0^T \int_\Omega \int_\RR \xi \partial_\xi b(\xi) \varphi(t,x) \Biggl( \int_{(-\infty,\xi]}
        \frac{\oP(t,x) - \Peff(\eta)}{\lambda+2\mu}
        \, \dd f(t,x,\eta) \Biggr) 
        \, \dd \xi \, \dd x \, \dd t\\
        &= 
        \int_0^T \int_\Omega \int_\RR \xi \partial_\xi b(\xi) \varphi(t,x) 
        \mathcal{M}[f](t,x,\xi) \, \dd \xi\, \dd x \, \dd t.
    \end{align*}
    Using these identities in $\eqref{Cor : Main Result : Equation to Calculate}$ we obtain $\eqref{Cor : Main Result : Eq on f}$.
\end{proof}

\section{Conclusions}\label{Sec : Conclusions}

We were concerned with the rigorous justification of an effective system for an isothermal compressible liquid-vapor flow in the regime where the number of phase boundaries is very large.
As a model on the detailed scale we have chosen the \system equations with a pressure function of Van-der-Waals type that can be seen as a lower-order approximation of the NSK equations.
For this system, we have proven a homogenization result that justifies an effective system in the regime where the initial density is highly oscillating.
This effective system describes the two-phase fluid on an averaged scale through two deterministic quantities in the form of the fluid's velocity and the order parameter and through a probabilistic quantity in the form of a Young measure.
Accordingly, the effective system of equations consists of a hydrodynamic part and a kinetic equation for the Young measure.
The kinetic equation coincides with the one in \cite{HB_NOTE}. However, in \cite{HB_NOTE}, the authors were able to rigorously reduce the effective system to a Baer--Nunziato type system by assuming that the Young measure is given initially by the convex combination of two Dirac-measures and by working in the stronger Hoff--Desjardin regularity class.
It would be very interesting whether such a reduction is also possible for the \system system with a pressure function of Van-der-Waals type. 
Also, by rewriting the kinetic equation for the Young measure into a kinetic equation for its corresponding cumulative distribution function we observed that the effective system is accessible by standard approximation methods.
It would be interesting to investigate the effective system by such methods without reducing it to a Baer--Nunziato type system.
\par
Our justification relies on the assumption that a finite energy weak solution to the initial-boundary value problem $\eqref{NSK_parabolic}$--$\eqref{NSK Parabolic Boundary Conditions}$ emanating from initial data that satisfy $\eqref{defi:renormalized:IC}$ exists globally in time.
Providing a rigorous proof of such a global-in-time existence result would complement our results and is thus planned as forthcoming work. 
In this contribution, we considered the isothermal case. 
A parabolic relaxation formulation for the non-isothermal NSK equations was recently proposed in \cite{KMR}.
It would be very nice to obtain a similar homogenization result for the non-isothermal case.
Our results rely on a parabolic relaxation formulation, however, there are other possibilities to approximate the NSK equations. 
In \cite{Rohde_Convolutional}, a non-local NSK model was proposed that models the capillarity term by a convolutional operator with an interaction potential.
For this model, a rigorous homogenization result for a porous domain was recently obtained in \cite{RLvW}.
In \cite{Rohde_Elliptic}, an elliptic relaxation formulation was proposed that corresponds to the \system equations with $\paracoup=0$. 
It would be very interesting to generalize our results to these NSK models and, also, to obtain a homogenization result for the parabolic NSK equations in a porous domain analogous to the one in \cite{RLvW}.
\appendix
\section{Compactness Results and Weak Continuity in Time}\label{Sec : Compactness Results and Weak Continuity in Time}
In this section we collect well-known compactness results for Bochner spaces and weakly continuous functions.
Also, we state a general form of the weak compactness result concerning the effective viscous flux for the compressible Navier--Stokes equations, that was stated and proven in \cite{NASI} in a even more general form.
These results are used throughout this work and therefore we include them here for the sake of clarity.
We start with the classical Aubin--Lions lemma.
\begin{lemma}\label{thm: Aubin Lions}
    Let $X_0,X,X_1$ be three Banach spaces satisfying the embeddings
    \begin{equation*}
        X_0 \hookrightarrow \hookrightarrow X \hookrightarrow X_1,
    \end{equation*}
    where the first embedding $X_0\hookrightarrow\hookrightarrow X$ is compact.\\
    For $p,q\in [1,\infty]$, define the space
    \begin{equation*}
        W_{p,q}:=\Bigl\{  u\in L^p(0,T;X_0) \mid \partial_t u \in L^q(0,T;X_1)  \Bigr\}
    \end{equation*}
    and a norm on $W_{p,q}$ via
    \begin{align*}
        \|u\|_{W_{p,q}} := \|u\|_{L^p(0,T;X_0)} + \|\partial_t u\|_{L^q(0,T;X_1)} 
        \quad \forall \,  u \in W_{p,q}.
    \end{align*}
    Then, we have that $\Bigl(W_{p,q},\|\cdot\|_{W_{p,q}}  \Bigr)$ is a Banach space and the embedding
    \begin{equation*}
        W_{p,q} \hookrightarrow\hookrightarrow L^p(0,T;X) 
    \end{equation*}
    is compact if $p \in [1,\infty)$.
    Moreover, the embedding
    \begin{equation*}
        W_{p,q} \hookrightarrow\hookrightarrow C([0,T];X)
    \end{equation*}
    is compact if $p = \infty$ and $q \in (1,\infty]$.
\end{lemma}
\begin{proof}
    See e.g.~\cite[Corollary 4]{S}.
\end{proof}

Another compactness result can be established for weakly continuous functions.
Let us start with the following standard result that ensures the weak continuity of a quantity.

\begin{lemma}\label{Lemma : Criterion for Weak Continuity in Time}
    Let $d \in \NN$ and $Q \subseteq \RR^d$ a bounded domain.
    Let $p \in (1,\infty)$ and suppose that $f \in L^\infty(0,T;L^p(Q))$ satisfies for any $\phi \in \mathcal{D}(\Omega)$
    \begin{equation*}
        \frac{\dd}{\dd t} \int_Q f\phi\,\dd x = F_\phi \quad \text{in } \mathcal{D}^\prime((0,T))
    \end{equation*}
    for some integrable function $F_\phi \in L^1(0,T)$.\\
    Then, there exists a weakly continuous function $\hat{f} \in C_{\mathrm{w}}([0,T];L^p(Q))$ such that
    \begin{equation*}
        \hat{f}(t) = f(t) \quad \text{in } L^p(Q) \quad \text{for almost all } t \in (0,T).
    \end{equation*}
    Moreover, we have for any $0\leq s \leq t\leq T$ the relation
    \begin{equation*}
        \int_Q \hat{f}(t)\phi\, \dd x = \int_Q \hat{f}(s) \phi \, \dd x
        +
        \int_s^t F_\phi(\tau) \, \dd \tau
    \end{equation*}
    for any test function $\phi \in \mathcal{D}(Q)$.
\end{lemma}
\begin{proof}
    See e.g.~\cite[Chapter 6]{NASI}.
\end{proof}
The following compactness result is a consequence of the Sobolev embedding theorem and an abstract version of the Arzéla--Ascoli theorem. 
For a detailed exposition of the following compactness result we refer to \cite[Chapter~6]{NASI}.
\begin{lemma}\label{Lemma : Weak Continuity in Time}
Let $T>0$, $d \in \NN$ with $d \geq 2$ and assume that $Q \subseteq \RR^d$ is a bounded domain with Lipschitz boundary. 
Let $p \in (1,\infty)$ and assume that functions $g_n \in C_{\mathrm{w}}([0,T];L^p(Q))$ are given for $n \in \NN$.\\
Then we have the following statements:
\begin{enumerate}
    \item If the sequence ${(g_n)}_{n\in\NN}$ is uniformly bounded in $L^p(Q)$, i.e.,
    \begin{equation*}
        \sup\limits_{n\in\NN}\sup\limits_{t \in [0,T]} \|g_n(t)\|_{L^p(Q)} \leq C
    \end{equation*}
    for some positive constant $C>0$ that does not depend on $n$, and uniformly continuous in $W^{-1,q}(Q)$ for some $q \in (1,\infty)$, i.e.,
    \begin{equation*}
        \forall \varepsilon >0: \exists \delta(\varepsilon)>0:
        \Bigl( t,s \in [0,T], \, |t-s|< \delta \Longrightarrow \| g_n(t)-g_n(s)\|_{W^{-1,q}(Q)} < \varepsilon \quad \forall \, n \in \NN \Bigr),
    \end{equation*}
    then there exists some $g \in C_{\mathrm{w}}([0,T];L^p(Q))$ such that, after passing to a non-relabeled subsequence,
    \begin{equation*}
        g_n \to g \quad\text{in } C_{\mathrm{w}}([0,T];L^p(Q)).
    \end{equation*}
    \item If there exists some $g \in C_{\mathrm{w}}([0,T];L^p(Q))$ such that
    \begin{equation*}
        g_n \to g \quad \text{in } C_{\mathrm{w}}([0,T];L^p(Q)),
    \end{equation*}
    then we have
    \begin{equation*}
        g_n \to g \quad \text{strongly in } L^r(0,T;W^{-1,s}(Q))
    \end{equation*}
    for any $r \in [1,\infty)$ and for any $s \in (1,\infty)$ such that the embedding $L^p(Q) \hookrightarrow\hookrightarrow W^{-1,s}(Q)$ is compact.
\end{enumerate}
\end{lemma}
\begin{proof}
    See e.g.~\cite[Chapter 6]{NASI}.
\end{proof}
We recall the following general version of the effective viscous flux lemma that was proven in \cite{NASI} in a more general setting.
\begin{theorem}\label{thm: Effective Viscous Flux}
    Let $T,\mu,\lambda>0$ and let $Q \subseteq \RR^3$ be a bounded domain. 
    Let $r,s \in (1,\infty)$ and $z \in (\frac{6}{5},\infty)$.
    For $n \in \NN$, let $\mathbf{q}_n \in C_{\mathrm{w}}([0,T];L^z(Q;\RR^3))$, $\uvec_n \in L^2(0,T;W^{1,2}_0(Q;\RR^3))$, $p_n \in L^r(0,T;L^r(Q))$, $\mathbf{F}_n \in L^s(0,T;L^s(Q;\RR^3))$, $f_n \in L^2(0,T;L^2(Q))$, $g_n \in L^\infty(0,T;L^\infty(Q))$ with $g_n \in C_{\mathrm{w}}([0,T];L^q(Q))$ for any $q \in [1,\infty)$ be given satisfying the convergences
    \begin{equation*}
        \begin{aligned}
            &\mathbf{q}_n \to \mathbf{q} \quad \text{in } C_{\mathrm{w}}([0,T];L^z(Q,\RR^3)),
            \quad 
            \uvec_n \weak \uvec \quad \text{in } L^2(0,T;W^{1,2}_0(Q;\RR^3))\\
            &p_n \weak p \quad \text{in } L^r(0,T;L^r(Q)), 
            \quad 
            \mathbf{F}_n  \weak \mathbf{F} \quad \text{in } L^s(0,T;L^s(Q;\RR^3)),\\
            &f_n \weak f \quad \text{in } L^2(0,T;L^2(Q)), 
            \quad 
            g_n \weakstar g \quad \text{in } L^\infty(0,T;L^\infty(Q)),\\
            & g_n \to g \quad \text{in } C_{\mathrm{w}}([0,T];L^q(Q)) 
            \quad \text{for any } q \in [1,\infty). 
        \end{aligned}
    \end{equation*}
    Assume that
    \begin{equation*}
        \partial_t \mathbf{q}_n 
        +
        \div(\mathbf{q}_n \otimes \uvec_n) 
        +
        \nabla p_n
        -
        \mu \Delta \uvec_n
        -
        (\lambda + \mu) \nabla( \div \uvec_n) 
        =
        \mathbf{F}_n 
        \quad \text{in } \mathcal{D}^\prime(Q_T;\RR^3)
    \end{equation*}
    and
    \begin{equation*}
        \partial_t g_n 
        +
        \div(g_n \uvec_n) 
        =
        f_n 
        \quad \text{in } \mathcal{D}^\prime(Q_T)
    \end{equation*}
    hold for any $n \in \NN$.\\
    Then we have the relation
    \begin{equation*}
        \begin{aligned}
            &\lim\limits_{n\to \infty} 
            \int_0^T \int_Q \eta(t) \phi(x) \biggl( p_n - (\lambda + 2\mu) \div\uvec_n \biggr)g_n\, \dd x \, \dd t\\
            &\qquad= \int_0^T \int_Q \eta(t) \phi(x) \biggl( p - (\lambda + 2\mu) \div \uvec\biggr) g \, \dd x \, \dd t
        \end{aligned}
    \end{equation*}
    for any test functions $\eta \in \mathcal{D}((0,T))$ and $\phi \in \mathcal{D}(Q)$.
\end{theorem}
\begin{proof}
    See \cite[Proposition 7.36]{NASI}. 
    There the authors prove a more general statement.
\end{proof}
We close this section with the following standard result that gives a criterion for which one can guarantee that the product of a weakly convergent and a strongly convergent quantity converges weakly to the product of the corresponding limits.
\begin{lemma}\label{lem: Product Convergences}
    Let $T>0$ and $Q\subseteq \RR^d$ be a domain. Let $p,q\in(1,\infty)$ and let $p^\prime,q^\prime \in (1,\infty)$ denote the conjugate Hölder exponents, i.e.,
    \begin{align*}
        p^\prime := \frac{p}{p-1}, \quad q^\prime:= \frac{q}{q-1}.
    \end{align*}
    For $n \in \NN$ let $f_n \in L^p(0,T;W^{1,q}(Q))$ and $g_n \in L^{p^\prime}(0,T;W^{-1,q^\prime}(Q))$ such that
    \begin{align*}
        f_n \weak f \quad \text{in } L^p(0,T;W^{1,q}(Q), \qquad
        g_n \to g \quad \text{in } L^{p^\prime}(0,T;W^{-1,q^\prime}(Q)).
    \end{align*}
    Then we have 
    \begin{align*}
        f_ng_n \weak fg \quad \text{in } \mathcal{D}^\prime((0,T)\times Q).
    \end{align*}
\end{lemma}
\begin{proof}
    Since $p,q \in (1,\infty)$, we have that
    \begin{align*}
        L^{p^\prime}(0,T;W^{-1,q^\prime}(Q)) \simeq \Bigl( L^p(0,T;W^{1,q}_0(Q)) \Bigr)^*.
    \end{align*}
    For $\varphi \in \mathcal{D}((0,T)\times Q)$ we thus have
    \begin{align*}
        &\langle f_n g_n - fg , \varphi \rangle_{\mathcal{D}^\prime,\,\mathcal{D}}
        =
        \langle f_n(g_n-g),\varphi\rangle_{\mathcal{D}^\prime,\,\mathcal{D}} 
        +
        \langle (f_n-f)g,\varphi \rangle_{\mathcal{D}^\prime,\,\mathcal{D}}
        \\
        &\leq
        \|g - g_n\|_{L^{p^\prime}(0,T;W^{-1,q^\prime}(Q))}\sup\limits_{n\in\NN}\|f_n\varphi\|_{L^p(0,T;W^{1,q}(Q))} \\
        &\qquad+  \langle g,(f_n-f)\varphi\rangle_{(L^p(0,T;W^{1,q}_0(\Omega))^\ast,\,L^p(0,T;W^{1,q}_0(\Omega))} \to 0
    \end{align*}
    as $n \to \infty$.
\end{proof}
\section*{Acknowledgements}
Funded by Deutsche Forschungsgemeinschaft (DFG, German Research Foundation) under Germany's Excellence Strategy - EXC 2075 - 390740016. We acknowledge the support by the Stuttgart Center for Simulation Science (SC SimTech).
C.~R.  acknowledges funding by the Deutsche Forschungsgemeinschaft (DFG, 
German Research Foundation) - SPP 2410 Hyperbolic Balance Laws in Fluid 
Mechanics: Complexity, Scales, Randomness (CoScaRa).
\bibliographystyle{plain}
\bibliography{sample}

\end{document}